\documentclass[08pt]{article}
\usepackage{graphicx} 
\usepackage[T1]{fontenc}
\usepackage[utf8]{inputenc}
\usepackage{lmodern}
\usepackage[a4paper,margin=1.5cm]{geometry}
\usepackage{amsmath,amssymb,amsfonts,amsthm}
\usepackage{bm}
\usepackage{hyperref}
\usepackage{mathtools}\usepackage{booktabs}
\usepackage{array}
\usepackage{multirow}
\usepackage{graphicx}
\usepackage{placeins}
\usepackage{float}

\theoremstyle{definition}
\newtheorem{definition}{Definition}[section]
\newtheorem{problem}[definition]{Problem}
\newtheorem{remark}[definition]{Remark}

\theoremstyle{plain}

\newtheorem{lemma}[definition]{lemma}
\usepackage{tikz-cd}
\usepackage{siunitx}
\usepackage{algorithm}
\usepackage{algpseudocode}
\usepackage{subcaption}
\usepackage{graphicx}
\usepackage{booktabs}
\usepackage{caption}
\usepackage{comment}

\usepackage{authblk}
\usepackage{hyperref}

\title{Stochastic Galerkin and Monte-Carlo methods for parabolic problems: Numerical performance of variational matrix-free approximations}
\author[1]{Moataz Dawor\thanks{Corresponding author: \href{mailto:daworm@hsu-hh.de}{daworm@hsu-hh.de}}}
\author[2]{Nils Margenberg}
\author[1]{Markus Bause}

\affil[1]{Helmut Schmidt University, University of the Federal Armed Forces Hamburg, Germany}
\affil[2]{University of Magdeburg, Institute for Analysis and Numerics, Magdeburg, Germany}

\begin{document}

\maketitle

\begin{abstract}
Stochastic Galerkin methods offer unexplored potential for the numerical simulation of parabolic problems with random variables, in particular if they are combined with variational discretizations of the space and time variables. Due to the high dimensionality, the solution of the arising algebraic systems do not become feasible without efficient solvers, preconditioners, and software architectures. A stochastic Galerkin discretization with an embedded slabwise finite element approximation of the space and time variables is proposed and analyzed numerically. For solving the linear systems, GMRES iterations are block-preconditioned by a geometric multigrid (GMG) technique using a local Vanka smoother for the space-time subsystems. Monte-Carlo methods are also used for solving random parabolic problems and studied here for the purpose of comparison. The Monte-Carlo approach is built on the space-time finite element formulation together with the GMRES-GMG solver technology. All algorithms have been implemented in a unified matrix-free framework based on the deal.II software library. Comparative numerical evaluations illustrate the performance properties of both approaches, including convergence of the discretizations and statistics of the algebraic solver. Superiority of the stochastic Galerkin approach is observed.
\end{abstract}

\section{Introduction}

Parabolic partial differential equations with random input data appear naturally when diffusion coefficients, source terms, or initial states are known only statistically. In such situations the numerical task is not only to approximate one solution, but to approximate statistical quantities of the whole random solution field, such as its mean and variance. Two computational strategies are especially natural for this purpose.

The first one is the Monte-Carlo method. It keeps the deterministic solver unchanged, draws independent samples of the random input, solves one deterministic parabolic problem for each realization, and forms empirical estimators of the desired quantities of interest. This makes Monte-Carlo robust, non-intrusive, and straightforward to parallelize. Its main drawback is also well known: the sampling error decays only with the root-sampling rate \(O(N_{\mathrm{MC}}^{-1/2})\), independently of how smooth the solution may be with respect to the random variables; see \cite{Caflisch1998MC,Sullivan2015UQ,Xiu2010Book}.

The second strategy is the intrusive stochastic Galerkin method. For Gaussian random variables, the stochastic dependence can be represented in an orthonormal Hermite polynomial chaos basis. Instead of solving many independent sample problems, one solves a coupled deterministic system for the retained Hermite modes. This coupling makes the method intrusive, but it also allows the method to exploit regularity in the stochastic space. In particular, if the solution has smooth or finite polynomial dependence on the random variables, the stochastic Galerkin approximation can converge much faster than standard sampling methods; see \cite{GhanemSpanos1991,XiuKarniadakis2002,Xiu2010Book,CohenDeVoreSchwab2010,SchwabGittelson2011,HoangSchwab2013}.

The aim of this work is to compare Monte-Carlo and stochastic Galerkin approximations for a stochastic heat equation in a setting where both methods use the same deterministic space-time finite element discretization. Thus, the comparison isolates the treatment of the stochastic variables: Monte-Carlo uses independent pathwise solves, whereas stochastic Galerkin solves one coupled tensor-product system for the retained Hermite modes.

The coupled stochastic Galerkin systems are large, since each time slab contains spatial finite element degrees of freedom, local time degrees of freedom, and stochastic chaos modes. We therefore avoid assembling the full coupled matrix and use a matrix-free operator application exploiting the tensor-product structure in space, time, and stochastic variables. The resulting systems are solved by flexible GMRES with stochastic block-Jacobi preconditioning, where each diagonal stochastic block is treated by the deterministic space-time geometric multigrid solver. This combines stochastic Galerkin block preconditioning \cite{RosseelVandewalle2010,Ullmann2010,BespalovLoghinYoungnoi2021} with matrix-free finite element technology in \texttt{deal.II} \cite{dealII97,kronbichlerKormann2012}.

The numerical study is based on a finite-chaos analytical prescribed solution. Since the exact Hermite structure is known, the stochastic Galerkin approximation has a clear expected behaviour: for \(p_\xi<q\), stochastic truncation error remains, whereas at \(p_\xi=q\) all exact stochastic modes are contained in the discrete space and the remaining error is due to the deterministic discretization and algebraic tolerance. The same benchmark also allows the Monte-Carlo error to be split into pathwise discretization and pure sampling contributions.

The paper is organized as follows. Section~\ref{sec:problem_statement} introduces the model problem and the stochastic Galerkin space-time discretization in section~\ref{sec:SG_discretization}. Section~\ref{sec:manufactured-benchmarks} presents the stochastic Galerkin results, including stochastic refinement, finite-chaos capture, solver statistics, and timing data. Section~\ref{subsec:mc-sampling-rate-diagnostics} summarizes the Monte-Carlo method and the corresponding error decomposition. The Monte-Carlo experiments and the direct work--accuracy comparison with stochastic Galerkin are given in Section~\ref{sec:comparison}.

\section{Problem statement}
\label{sec:problem_statement}
We begin by introducing the heat equation with random inputs and the functional setting in which it is posed. Let \(D\subset\mathbb R^d\), with \(d\in \{2,3\}\), be a bounded Lipschitz domain, let \(I\coloneqq (0,T]\) for some \(T>0\), and let \((\Omega,\mathcal F,\mathbb P)\) be a complete probability space. We consider a random vector
\begin{equation}
    \bm{\xi}=(\xi_1,\dots,\xi_M):\Omega\to\Gamma=\mathbb R^M, \qquad \xi_m\sim\mathcal N(0,1),\qquad m=1,\dots,M \ ; \qquad \rho(\bm{\xi}) = \prod_{m=1}^M \frac{1}{\sqrt{2\pi}}e^{-\xi_m^2/2},
    \label{eq:random_varibale_law}
\end{equation}
where \(M\in\mathbb N\) denotes the number of random input variables in the model. We assume that the components \(\xi_1,\dots,\xi_M\) are independent and identically distributed standard Gaussian random variables, and $\rho(\bm{\xi})$ is the law of \(\bm{\xi}\), which admits the density with respect to the Lebesgue measure on \(\Gamma=\mathbb R^M\).\\

\begin{remark}
In this work we restrict ourselves to a finite-dimensional Gaussian setting with fixed \(M\). More general situations, such as dependent, non-identically distributed, or infinite-dimensional random inputs, are also of interest, but lie outside the scope of the present study.\\
\end{remark}

We consider the heat equation with random inputs:
\begin{align}
\partial_t u(\bm{x},t,\bm{\xi})
-\nabla\cdot\!\bigl(a(\bm{x},\bm{\xi})\nabla u(\bm{x},t,\bm{\xi})\bigr)
&=f(\bm{x},t,\bm{\xi})
&&\text{in } D\times I\times \Gamma,
\label{eq:heat_strong}
\\
u(\bm{x},t,\bm{\xi})&=0
&&\text{on } \partial D\times I\times \Gamma,
\label{eq:heat_bc}
\\
u(0,\bm{x},\bm{\xi})&=u_0(\bm{x},\bm{\xi})
&&\text{in } D\times \Gamma.
\label{eq:heat_ic}
\end{align}

Equivalently, upon viewing \(\bm{\xi}\) as a parameter, the problem may be interpreted as a deterministic parametric PDE posed on the product domain $D\times I\times \Gamma.$

\subsection{Deterministic spaces, stochastic Bochner spaces, and norms}

We now introduce the continuous function spaces underlying the weak formulation and indicate how they relate to the discrete approximation used later. In the present work, we employ a conforming finite element discretization in space, a discontinuous Galerkin discretization in time, and a stochastic Galerkin approximation in a finite-dimensional Hermite chaos space. Accordingly, the functional setting must simultaneously account for spatial weak derivatives, weak time derivatives, and square-integrability with respect to the probability density of the random inputs.

\medskip
\noindent
\textbf{Deterministic spatial spaces.}
We define $H:=L^2(D)$, $V:=H_0^1(D)$, $V':=H^{-1}(D)$. Thus \(H\) is the basic state space, \(V\) is the natural energy space for the diffusion operator, and \(V'\) is its dual space. The norm on \(V\) is given by $\|v\|_V:=\|\nabla v\|_{L^2(D)}$, which, by Poincar\'e's inequality, is equivalent to the full \(H^1(D)\)-norm on \(H_0^1(D)\). Hence \(\|\cdot\|_V\) is the natural energy norm associated with the elliptic part of the heat equation.

\medskip
\noindent
\textbf{Stochastic Bochner spaces.}
Let \(X\) be a separable Hilbert space. Since the random vector \(\bm{\xi}\) has density \(\rho\) on \(\Gamma\), the natural stochastic analogue of \(L^2\) is the weighted Bochner space
\[
L_\rho^2(\Gamma;X)
:=
\left\{
w:\Gamma\to X\ \text{strongly measurable}:
\int_\Gamma \|w(\bm{\xi})\|_X^2\,\rho(\bm{\xi})\,\mathrm d\bm{\xi}<\infty
\right\}.
\]
This space contains \(X\)-valued random fields with finite second moment with respect to the measure \(\rho(\bm{\xi})\,\mathrm d\bm{\xi}\). It is equipped with the inner product
\[
\langle w,z \rangle_{L_\rho^2(\Gamma;X)}
=
\int_\Gamma  \langle w(\bm{\xi}),z(\bm{\xi}) \rangle_X\,\rho(\bm{\xi})\,\mathrm d\bm{\xi}
=
\mathbb E_\rho\!\bigl[\langle w,z \rangle_X\bigr],
\ \text{ where } \
\mathbb E_\rho[g]
:=
\int_\Gamma g(\bm{\xi})\,\rho(\bm{\xi})\,\mathrm d\bm{\xi}.
\]
The induced norm is
\(
\|w\|_{L_\rho^2(\Gamma;X)}
=
\left(
\int_\Gamma \|w(\bm{\xi})\|_X^2\,\rho(\bm{\xi})\,\mathrm d\bm{\xi}
\right)^{1/2}.
\)
In particular, choosing \(X=H\) or \(X=V\) allows us to measure random fields in the state norm or in the energy norm, respectively.

\medskip
\noindent
\textbf{Time-dependent random fields.}
For random fields depending on both time and the stochastic variable, we use
\[
L_\rho^2(\Gamma;L^2(I;X))
=
\left\{
w:\Gamma\to L^2(I;X):
\int_\Gamma\int_0^T \|w(t,\bm{\xi})\|_X^2\,\mathrm dt\,\rho(\bm{\xi})\,\mathrm d\bm{\xi}<\infty
\right\}.
\]
Its norm is
\[
\|w\|_{L_\rho^2(\Gamma;L^2(I;X))}
=
\left(
\int_\Gamma\int_0^T \|w(t,\bm{\xi})\|_X^2\,\mathrm dt\,\rho(\bm{\xi})\,\mathrm d\bm{\xi}
\right)^{1/2}.
\]
By Fubini's theorem, this space may equivalently be viewed as \(L^2(I;L_\rho^2(\Gamma;X))\). This identification will be used repeatedly in the variational formulation and in the error analysis.

\medskip
\noindent
\textbf{Continuous parabolic solution space.}
The natural solution space for the weak form of the random heat equation is
\[
\mathcal X
:=
\Bigl\{
w\in L_\rho^2(\Gamma;L^2(I;V))
:
\partial_t w\in L_\rho^2(\Gamma;L^2(I;V'))
\Bigr\}.
\]
This is the standard parabolic graph space: the first term requires finite spatial energy, while the second requires the time derivative only in the dual space \(V'\), which is the natural regularity level for weak solutions of parabolic problems. We equip \(\mathcal X\) with the graph norm
\begin{equation}
\|w\|_{\mathcal X}^2
:=
\|w\|_{L_\rho^2(\Gamma;L^2(I;V))}^2
+
\|\partial_t w\|_{L_\rho^2(\Gamma;L^2(I;V'))}^2
+
\|w(0)\|_{L_\rho^2(\Gamma;H)}^2.
\label{eq:X_norm}
\end{equation}
The initial trace \(w(0)\in L_\rho^2(\Gamma;H)\) is included because the heat equation is an evolution problem and the initial value enters the weak formulation explicitly; see \cite{Thomee2006,BauseRaduKoecher2017}.

\medskip
\noindent
\textbf{Discrete-in-time viewpoint and broken norm.}
Later, the time variable will be discretized by a discontinuous Galerkin method on a partition
\[
0=t_0<t_1<\cdots<t_N=T,\qquad I_n:=(t_{n-1},t_n],\qquad \tau_n:=t_n-t_{n-1}.
\]
Since dG functions are only piecewise polynomial in time and may be discontinuous at the nodes \(t_n\), the continuous graph norm \(\|\cdot\|_{\mathcal X}\) is no longer the appropriate stability norm on the discrete level. For this reason, we introduce the broken-in-time dG norm; see \cite{SchotzauSchwab2000,BauseKoecher2015}.
\begin{equation}
\|w\|_{\mathrm{dG}}^2
:=
\sum_{n=1}^{N}
\|w\|_{L_\rho^2(\Gamma;L^2(I_n;V))}^2
+
\sum_{n=1}^{N-1}\|[w]_n\|_{L_\rho^2(\Gamma;H)}^2
+
\|w(0^+)\|_{L_\rho^2(\Gamma;H)}^2, \ [w]_n:=w(t_n^+)-w(t_n^-)
\label{eq:dg_energy_norm}
\end{equation}
denotes the jump across the interface \(t_n\).\\

This norm is the natural discrete counterpart of the continuous parabolic norm. The first term measures the spatial energy on each time slab, the jump terms penalize inter-slab discontinuities, and the last term controls the incoming initial value on the first slab. In particular, if a function is continuous in time, then all jumps vanish and the broken norm reduces to the slabwise energy contribution together with the initial trace term.

\medskip
\begin{remark}
The continuous spaces introduced above are chosen precisely so that the later discrete spaces fit into them naturally. Indeed, the conforming spatial finite element space \(V_h\) satisfies \(V_h\subset V\), so spatially discrete functions inherit the same \(H\)- and \(V\)-norms as the continuous problem. Likewise, the stochastic Galerkin space \(\mathcal{S}_{p_\xi}\subset L_\rho^2(\Gamma)\) inherits the stochastic norm from \(L_\rho^2(\Gamma)\). On the other hand, the discontinuous Galerkin time space consists of functions that are only piecewise polynomial in time, so these functions need not belong to \(\mathcal X\) globally; this is exactly why the broken norm \(\|\cdot\|_{\mathrm{dG}}\) is needed.
\end{remark}

Consequently, the continuous norm \(\|\cdot\|_{\mathcal X}\) governs the weak formulation and continuous well-posedness, while the broken norm \(\|\cdot\|_{\mathrm{dG}}\) governs the stability and error analysis of the fully discrete space-time approximation. In this sense, the discrete norms are not ad hoc: they are designed to mirror the structure of the continuous parabolic problem as closely as possible while accounting for the time discontinuities introduced by the dG discretization.\\

\subsection{Weak formulation}

We now formulate the random heat equation in weak form. Since the spatial operator is elliptic supplemented with homogeneous Dirichlet boundaries, the natural energy space is \(V=H_0^1(D)\), while the time derivative is understood in the dual space \(V'=H^{-1}(D)\), \textit{cf.} \cite{Thomee2006,BauseRaduKoecher2017}.

To ensure well-posedness, we assume that the diffusion coefficient is essentially bounded and uniformly elliptic. More precisely, we assume that $a\in L^\infty(D\times\Gamma)$ and that there exist deterministic constants \(a_{\min},a_{\max}>0\) such that
\begin{equation}
0<a_{\min}\le a(\bm x,\bm\xi)\le a_{\max}<\infty
\qquad
\text{for a.e. }(\bm x,\bm\xi)\in D\times\Gamma
\label{eq:uniform_ellipticity}
\end{equation}
with respect to the product measure \(d\bm x\otimes \rho(\bm\xi)\,d\bm\xi\). Thus the pathwise diffusion operator is uniformly elliptic for \(\mathbb P\)-almost every realization.

For each fixed \(\bm\xi\in\Gamma\), we define the bilinear form
\begin{equation}\label{diff_bilinear}
\mathfrak a(\bm\xi;w,v)
:=
\int_D a(\bm x,\bm\xi)\,\nabla w(\bm x)\cdot\nabla v(\bm x)\,\mathrm d\bm x,
\qquad w,v\in V.
\end{equation}
By \eqref{eq:uniform_ellipticity}, this bilinear form is bounded and coercive uniformly in \(\bm\xi\), namely $|\mathfrak a(\bm\xi;w,v)| \le a_{\max}\|w\|_V\|v\|_V$, $\mathfrak a(\bm\xi;v,v) \ge a_{\min}\|v\|_V^2$ \ , \textit{cf.} \cite{Sullivan2015UQ,Thomee2006,NobileTempone2009,BauseRaduKoecher2017}. These estimates are the key ingredients for the stochastic parabolic well-posedness in the space
\[
\mathcal X
=
\bigl\{
w\in L_\rho^2(\Gamma;L^2(I;V))
:
\partial_t w\in L_\rho^2(\Gamma;L^2(I;V'))
\bigr\}.
\]

\begin{problem}[Weak stochastic heat problem]
\label{prob:weak-stochastic-heat}
Given $u_0\in L_\rho^2(\Gamma;H)$ and $f\in L_\rho^2(\Gamma;L^2(I;V'))$,
find $u\in \mathcal X$ with $u(0)=u_0$ in $L_\rho^2(\Gamma;H)$
such that
\begin{equation}
\begin{aligned}
\int_\Gamma\int_0^T
\Bigl(
\langle \partial_t u(\bm{\cdot},t,\bm{\xi}),v(\bm{\cdot},t,\bm{\xi})\rangle_{V',V}
& +
\mathfrak a(\bm\xi;u(\bm{\cdot},t,\bm{\xi}),v(\bm{\cdot},t,\bm{\xi}))
\Bigr)
\,\mathrm dt\,\rho(\bm\xi)\,\mathrm d\bm\xi\\
& =
\int_\Gamma\int_0^T
\langle f(\bm{\cdot},t,\bm{\xi}),v(\bm{\cdot},t,\bm{\xi})\rangle_{V',V}
\,\mathrm dt\,\rho(\bm\xi)\,\mathrm d\bm\xi
\end{aligned}
\label{eq:weak_problem}
\end{equation}
for all test functions $v\in L_\rho^2(\Gamma;L^2(I;V)).$
\end{problem}

In expectation notation, \eqref{eq:weak_problem} can be written equivalently as
\[
\mathbb E_\rho\!\left[
\int_0^T
\Bigl(
\langle \partial_t u,v\rangle_{V',V}
+
\mathfrak a(\bm\xi;u,v)
\Bigr)\,\mathrm dt
\right]
=
\mathbb E_\rho\!\left[
\int_0^T
\langle f,v\rangle_{V',V}\,\mathrm dt
\right].
\]

This formulation is the stochastic counterpart of the classical weak formulation of a deterministic parabolic problem. It combines the weak-in-space, weak-in-time structure of the heat equation with square-integrability in the stochastic variables. Moreover, for \(\rho\)-almost every \(\bm\xi\in\Gamma\), the problem may be interpreted pathwise as a deterministic parabolic PDE in \((\bm{x},t)\). The stochastic weak formulation above is therefore obtained by integrating the pathwise weak formulation over the parameter domain \(\Gamma\) with respect to the probability measure \(\rho(\bm\xi)\,d\bm\xi\).\\

\subsection{Discrete approximation strategy: space-time--stochastic structure}
\label{sec:SG_discretization}

The weak solution of the random heat equation depends on three variables of different analytical nature: the spatial variable \(\bm x\in D\), the time variable \(t\in I\), and the random parameter \(\bm\xi\in\Gamma\). It is therefore natural to approximate the solution by combining three discretization principles, each adapted to one of these variables.

In space, we use a conforming finite element discretization in order to preserve the variational structure of the elliptic operator and the homogeneous Dirichlet boundary condition. In time, we use a discontinuous Galerkin method as in \cite{NMP_hp,NP_parabolic}, which is particularly convenient for parabolic evolution problems because it is local in time, naturally accommodates slabwise solution procedures, and allows for controlled jumps at time interfaces. In the stochastic variable, we use a generalized polynomial chaos approximation in a Hermite basis, which is the canonical orthonormal choice for Gaussian random inputs.

This leads to a fully discrete approximation of tensor-product type, $\text{space} \;\otimes\; \text{time} \;\otimes\; \text{stochastic},$ and, more precisely, to a discrete solution space of the form $X_{h,\tau,p_\xi} = V_h \otimes S_\tau \otimes \mathcal P_{p_\xi}$ as studied, for related space-time and stochastic Galerkin constructions, in \cite{Thomee2006,SchotzauSchwab2000,XiuKarniadakis2002,HoangSchwab2013}.
The purpose of the next three subsections is to define these factors one by one and to explain why they are the natural choices in the present stochastic parabolic setting. This is also the structure underlying the later slabwise stochastic Galerkin formulation and its tensor-product algebraic form.

\subsection{Spatial finite element discretization}
\label{subsec:sg-spatial-discretization}
Let \(\mathcal T_h=\{K\}\) be a shape-regular conforming mesh of \(D\)
consisting of quadrilateral or hexahedral elements. We choose a conforming
finite element space \(V_h\subset V=H_0^1(D)\). For polynomial degree
\(k\in\mathbb N\), we define
\[
V_h
\coloneqq
\left\{
v_h\in C^0(\overline D)\cap H_0^1(D)
:
v_h|_K\in V_k(K)
\quad \forall K\in\mathcal T_h
\right\},
\]
where the local physical-cell space \(V_k(K)\) is defined by the pullback from
the reference cell,
\[
V_k(K)
\coloneqq
\left\{
\widehat v\circ F_K^{-1}
:
\widehat v\in \mathbb Q_k(\widehat K)
\right\}.
\]
Here \(F_K:\widehat K\to K\) denotes the element map from the reference cell
\(\widehat K\) to the physical cell \(K\), and
\(\mathbb Q_k(\widehat K)\) is the tensor-product polynomial space of degree
at most \(k\) in each coordinate.

Thus \(V_h\) consists of globally continuous, piecewise polynomial functions that vanish on \(\partial D\). The inclusion \(V_h\subset H_0^1(D)\) is crucial: it ensures that the spatially discrete problem inherits the same energy structure as the continuous weak formulation.

We fix a basis $V_h=\operatorname{span}\{\phi_m\}_{m=1}^{N_h}$, $\dim V_h=N_h$. In a standard conforming finite element implementation, the functions \(\phi_m\) are the global nodal Lagrange basis functions associated with the chosen mesh and polynomial degree. If one wishes, the local polynomial space on the reference cell may also be viewed from a modal perspective, for example in terms of tensorized Legendre polynomials; however, for the present conforming formulation and its later assembly, the global Lagrange basis is the natural representation.

\begin{remark}
Each coefficient in the expansion
$v_h(\bm x)=\sum_{m=1}^{N_h} V_m \phi_m(\bm x)$ represents one spatial degree of freedom. The basis \(\{\phi_m\}\) will later form the spatial tensor factor in the fully discrete representation of the stochastic space-time solution.
\end{remark}

\subsection{Broken polynomial spaces in time and the \(\mathrm{dG}(r)\) discretization}
\label{subsec:sg-time-discretization}

We partition the time interval \(I=(0,T]\) into slabs
\(0=t_0<t_1<\cdots<t_N=T\), with \(I_n:=(t_{n-1},t_n]\) and
\(\tau_n:=t_n-t_{n-1}\). In the present space-time finite element setting,
the time discretization is based on scalar slabwise polynomials combined with the spatial finite element space. We define
\[
S_\tau := S_\tau^{(r)} :=
\left\{
\theta: [0,T] \to\mathbb R:\;
\theta|_{I_n}\in \mathbb P_r(I_n),
\; \forall n=1,\ldots,N
\right\}.
\]
The corresponding \(V_h\)-valued time-discrete space is represented through the tensor-product identification
\(
T_\tau^{(r)}(V_h)
\cong
S_\tau\otimes V_h .
\)
Thus a discrete space-time function is polynomial in time on every slab, with coefficients in \(V_h\), but it is not required to be continuous across slab interfaces. We therefore use the one-sided traces
\[
\omega(t_n^-):=\lim_{t\nearrow t_n}\omega(t),
\qquad
\omega(t_n^+):=\lim_{t\searrow t_n}\omega(t),
\qquad
[\omega]_n:=\omega(t_n^+)-\omega(t_n^-).
\]
These jumps enter the discontinuous Galerkin formulation and provide the
causal transfer of information from one time slab to the next. On each slab
\(I_n\), we choose a local basis
\(
\{v_i^{(n)}\}_{i=0}^{r}\subset \mathbb P_r(I_n)
\)
for the scalar polynomial space. In the computations, we use the right Gauss--Radau nodal time basis on each slab \(I_n\); see \cite{BauseKoecher2015}. Since this basis includes the right endpoint of the slab, it is convenient for evaluating the outgoing trace and for assembling the endpoint contribution in the \(\mathrm{dG}(r)\) formulation. The resulting slab-local structure is one of the reasons why discontinuous Galerkin time discretizations are well suited to parabolic evolution problems; \textit{cf.} \cite{SchotzauSchwab2000,BauseKoecher2015}.

\subsection{Stochastic discretization by Hermite generalized polynomial chaos}

We now turn to the stochastic variable. The natural ambient space is the weighted Hilbert space
\begin{equation}
L_\rho^2(\Gamma)
:=
\left\{
g:\Gamma\to \mathbb R:\;
\int_\Gamma |g(\bm\xi)|^2\rho(\bm\xi)\,\mathrm d\bm\xi<\infty
\right\},  \ \ \text{equipped with} \ \  \langle g,h\rangle_{L_\rho^2(\Gamma)}
:=
\int_\Gamma g(\bm\xi)h(\bm\xi)\rho(\bm\xi)\,\mathrm d\bm\xi.
\end{equation}
Since the random inputs are independent standard Gaussian variables, the Wiener--Askey scheme shows that the natural generalized polynomial chaos basis is given by Hermite polynomials; see \cite{GhanemSpanos1991,XiuKarniadakis2002,Xiu2010Book}. More precisely, for one standard Gaussian variable \(\xi\sim\mathcal N(0,1)\), we consider the probabilists' Hermite polynomials
\begin{equation}
\mathrm{He}_n(y)
=
(-1)^n e^{y^2/2}\frac{d^n}{dy^n}\bigl(e^{-y^2/2}\bigr),
\qquad
\psi_n(y):=\frac{\mathrm{He}_n(y)}{\sqrt{n!}} \, \qquad n\in\mathbb N_0,
\label{eq:Hermit_def}
\end{equation}
Then \(\{\psi_n\}_{n\ge 0}\) is an orthonormal basis of \(L^2(\mathbb R,\rho)\), that is, \( \int_{\mathbb R}\psi_n(y)\psi_m(y)\rho(y)\,\mathrm dy = \delta_{nm}.\). For the multivariate Gaussian variable \(\bm\xi=(\xi_1,\dots,\xi_M)\), we introduce the tensorized Hermite basis
\[
\Psi_{\bm{\alpha}}(\bm\xi)
:=
\prod_{m=1}^M \psi_{\alpha_m}(\xi_m),
\qquad
\bm{\alpha}=(\alpha_1,\dots,\alpha_M)\in\mathbb N_0^M,
\
\ \text{We also write } \
|\bm{\alpha}|_1:=\alpha_1+\cdots+\alpha_M \ ,
\ \ 
\bm{\alpha}!
:=
\prod_{m=1}^{M}\alpha_m!
\]
Because the variables \(\xi_1,\dots,\xi_M\) are independent and identically distributed, the product structure of the density implies \(\rho(\bm\xi)=\prod_{m=1}^M \rho(\xi_m),\) and therefore the orthonormality of the multivariate basis factorizes into the one-dimensional orthonormality relations:
\begin{equation}
\mathbb E_\rho[\Psi_{\bm{\alpha}}\Psi_{\bm{\beta}}]
=
\int_\Gamma \Psi_{\bm{\alpha}}(\bm\xi)\Psi_{\bm{\beta}}(\bm\xi)\rho(\bm\xi)\,\mathrm d\bm\xi
=
\prod_{m=1}^M \int_{\mathbb R}\psi_{\alpha_m}(y)\psi_{\beta_m}(y)\rho(y)\,\mathrm dy
=
\delta_{\bm{\alpha}\bm{\beta}}.
\label{eq:hermite-orthonormality}
\end{equation}
Hence \(\{\Psi_{\bm{\alpha}}\}_{\bm{\alpha}\in\mathbb N_0^M}\) forms an orthonormal basis of \(L_\rho^2(\Gamma)\).

The reason for choosing Hermite chaos is exactly this orthonormality with respect to the multivariate Gaussian density \(\rho\): it makes the stochastic approximation compatible with the underlying probability law and leads to a particularly clean Galerkin structure. In particular, the stochastic mass matrix becomes the identity, while the remaining stochastic coupling is carried by triple-product matrices that appear later in the tensor-product algebraic formulation; \textit{cf.} \cite{Sullivan2015UQ,XiuKarniadakis2002,Xiu2010Book,HoangSchwab2013}.

For a prescribed stochastic polynomial degree \(p_\xi\in\mathbb N_0\), we define the total-degree chaos space
\[
\mathcal P_{p_\xi}
:=
\operatorname{span}\{\Psi_{\bm{\alpha}}:\;|\bm{\alpha}|_1\le p_\xi\},
\ \text{ its dimension is: } \
N_\xi(p_\xi)
=
\dim \mathcal P_{p_\xi}
=
\binom{M+p_\xi}{p_\xi}.
\]
For instance, if \(M=4\), then \( N_\xi(0)=1, \ N_\xi(1)=5,\ N_\xi(2)=15,\ N_\xi(3)=35,\ N_\xi(4)=70.\)

\subsection{Fully discrete stochastic space-time finite element space}
\label{subsec:sg-fully-discrete-space}

Combining the spatial, temporal, and stochastic discretizations, we define the fully discrete approximation space by
\begin{equation}
\label{eq:global_disc_solution_space}
X_{h,\tau,p_\xi}
:=
V_h\otimes S_\tau\otimes \mathcal P_{p_\xi}.
\end{equation}
Hence the discrete solution is built from three independent approximation mechanisms: a conforming spatial basis in \(V_h\), a broken polynomial basis in time, and a Hermite chaos basis in the stochastic variable. The ordering in \eqref{eq:global_disc_solution_space} follows the implementation layout used in the code.

\begin{remark}
Let \(\{v_i^{(n)}\}_{i=0}^r\) be the local scalar time basis on \(I_n\), let \(\{\phi_m\}_{m=1}^{N_h}\) be the spatial finite element basis of \(V_h\), and let \(\{\Psi_{\bm{\alpha}}\}_{|\bm{\alpha}|_1\le p_\xi}\) be the stochastic chaos basis of \(\mathcal P_{p_\xi}\). Then every discrete function \(u_{h,\tau,p_\xi}\in X_{h,\tau,p_\xi}\) admits on each slab \(I_n\) the representation
\[
{u_{h,\tau,p_\xi}}\big|_{I_n}(\bm{x},t,\bm{\xi})
=
\sum_{i=0}^r
\sum_{|\bm{\alpha}|_1\le p_\xi}
\sum_{m=1}^{N_h}
U_{i,\bm{\alpha},m}^{(n)}
\,v_i^{(n)}(t)\,
\phi_m(\bm x)\,
\Psi_{\bm{\alpha}}(\bm\xi).
\]
The coefficients \(U_{i,\bm{\alpha},m}^{(n)}\) are precisely the unknowns of the slabwise stochastic Galerkin system.
\end{remark}

\begin{remark}
The space \(X_{h,\tau,p_\xi}\) is the natural discrete counterpart of the continuous stochastic parabolic setting:
\begin{itemize}
\item the factor \(V_h\subset H_0^1(D)\) preserves the spatial energy structure;
\item the factor \(S_\tau\) realizes a slabwise \(\mathrm{dG}(r)\) approximation in time;
\item the factor \(\mathcal P_{p_\xi}\subset L_\rho^2(\Gamma)\) realizes the stochastic Galerkin approximation with respect to the Gaussian law.
\end{itemize}
Accordingly, the later discrete bilinear forms and algebraic matrices inherit a tensor-product structure in space, time, and stochastic modes.
\end{remark}

\subsection{Slabwise dG stochastic Galerkin formulation}
\label{subsec:sg-slabwise-dg-formulation}

We now combine the spatial finite element discretization, the broken-in-time \(\mathrm{dG}(r)\) discretization, and the Hermite generalized polynomial chaos approximation into the fully discrete stochastic Galerkin formulation.

The global discrete trial and test space is \(X_{h,\tau,p_\xi}=V_h\otimes S_\tau\otimes \mathcal P_{p_\xi}\). Thus the discrete solution is sought globally in a broken-in-time tensor-product space. However, because the discontinuous Galerkin formulation in time couples only neighboring slabs through the interface traces, the resulting algebraic system is causal and block lower triangular in time; \textit{cf.} \cite{SchotzauSchwab2000,BauseKoecher2015}. Consequently, although the formulation is global, it can be solved successively slab by slab.

More precisely, on each slab \(I_n=(t_{n-1},t_n]\), we introduce the local discrete space
\(
X_{h,\tau,p_\xi}^{(n)}
:=
\mathbb P_r(I_n;V_h)\otimes \mathcal P_{p_\xi}.
\)
Hence the slabwise unknown \(u_{h,\tau,p_\xi}|_{I_n}\) belongs to the correct local tensor-product space consisting of \(V_h\)-valued polynomials in time on \(I_n\), combined with stochastic chaos modes up to degree \(p_\xi\).\\

To simplify notation, for \(\bm{\xi}\in\Gamma\) and \(w,v\in V_h\), we use bilinear form $\mathfrak a(\bm\xi;w,v)$ as in \ref{diff_bilinear}. Starting from the global broken variational formulation and restricting it to a single slab \(I_n\), one obtains the following causal local problem.

\begin{problem}[Discrete problem on \(I_n\)]
\label{prob:discrete-slab-problem}
Given the incoming trace \(u_{h,\tau,p_\xi}(\bm{x},t_{n-1}^-,\bm{\xi})\),
find \(u_{h,\tau,p_\xi}|_{I_n}\in X_{h,\tau,p_\xi}^{(n)}\) such that for all
\(v_{h,\tau,p_\xi}\in X_{h,\tau,p_\xi}^{(n)}\) there holds
\begin{equation}
\begin{aligned}
&
\mathbb E_\rho\!\left[
\int_{I_n}
\Bigl(
\left\langle
\partial_t u_{h,\tau,p_\xi},\,v_{h,\tau,p_\xi}
\right\rangle_{H}
+
\mathfrak{a}_h(\bm{\xi};u_{h,\tau,p_\xi},v_{h,\tau,p_\xi})
\Bigr)\,\mathrm dt
\right] +
\mathbb E_\rho\!\left[
\left\langle
u_{h,\tau,p_\xi}(\bm{x},t_{n-1}^+,\bm{\xi}),
\,v_{h,\tau,p_\xi}(\bm{x},t_{n-1}^+,\bm{\xi})
\right\rangle_{H}
\right]
\\
&=
\mathbb E_\rho\!\left[
\int_{I_n}
\left\langle
f(\bm{x},t,\bm{\xi}),\,v_{h,\tau,p_\xi}
\right\rangle_{H}
\,\mathrm dt
\right]
+
\mathbb E_\rho\!\left[
\left\langle
u_{h,\tau,p_\xi}(\bm{x},t_{n-1}^-,\bm{\xi}),
\,v_{h,\tau,p_\xi}(\bm{x},t_{n-1}^+,\bm{\xi})
\right\rangle_{H}
\right].
\end{aligned}
\label{eq:slab_discrete_problem}
\end{equation}
\end{problem}

Here the term involving \(u_{h,\tau,p_\xi}(\bm{x},t_{n-1}^+,\bm{\xi})\) belongs to the current slab \(I_n\), whereas the term involving \(u_{h,\tau,p_\xi}(\bm{x},t_{n-1}^-,\bm{\xi})\) transfers the outgoing information from the previous slab \(I_{n-1}\). In this way, the discontinuous Galerkin interface term weakly enforces time propagation from one slab to the next.

On the first slab, there is no previous time interval. Therefore the incoming trace is replaced by the discrete initial datum, $u_{h,\tau,p_\xi}(\bm{x},t_0^-,\bm{\xi})=u_{0,h,p_\xi}(\bm{x},\bm{\xi}),$ where \(u_{0,h,p_\xi}\in V_h\otimes\mathcal P_{p_\xi}\) denotes the chosen projection of the initial value \(u_0\) onto the spatial--stochastic discrete space.

Hence the global broken-space formulation may be interpreted algorithmically as follows: on each slab \(I_n\), one solves for the restriction \(u_{h,\tau,p_\xi}|_{I_n}\in X_{h,\tau,p_\xi}^{(n)}\), using as data either the projected initial condition (for \(n=1\)) or the incoming trace from the previous slab (for \(n\ge 2\)). This causal structure is precisely what later leads to the block lower triangular algebraic system in time.

\subsection{Tensor-product algebraic formulation}

We now derive the algebraic form of the slabwise stochastic Galerkin discretization. The purpose of this step is to make explicit how the fully discrete tensor-product space in (\ref{eq:global_disc_solution_space}) induces a corresponding tensor-product structure in the discrete matrices.

The key observation is that, on each slab \(I_n\), both trial and test functions are expanded in basis functions of the form \(v_i^{(n)}(t)\,\phi_m(\bm x)\,\Psi_{\bm{\alpha}}(\bm\xi),\) that is, as tensor products of a temporal basis function, a spatial finite element basis function, and a stochastic chaos basis function. When these expansions are inserted into the slabwise discrete variational problem, the arising integrals separate into products of temporal, spatial, and stochastic contributions. This is the origin of the Kronecker-product structure in the algebraic system; \textit{cf.} \cite{RosseelVandewalle2010,Ullmann2010,BespalovLoghinYoungnoi2021}.

A further ingredient is the stochastic expansion of the diffusion coefficient. Since the elliptic operator contains \(a(\bm x,\bm\xi)\) multiplicatively, the stochastic Galerkin projection of the diffusion term couples the stochastic modes of the solution with the stochastic modes of the coefficient. To expose this coupling explicitly, we expand the diffusion coefficient in the same Hermite basis as used for the solution.

\paragraph{Expansion of the diffusion coefficient.}
We write
\begin{equation}
a(\bm x,\bm\xi)
=
\sum_{\bm{\mu}\in \mathcal M_a} a_{\bm{\mu}}(\bm x)\Psi_{\bm{\mu}}(\bm\xi),
\label{eq:coeff_expansion}
\end{equation}
where \(\mathcal M_a\subset \mathbb N_0^M\) denotes the active coefficient index set. For the prescribed polynomial finite order diffusion used below, \(\mathcal M_a\) is finite and contains only modes of total degree at most \(2\).

\begin{remark}
It is important to distinguish between the stochastic index set of the discrete solution and the active stochastic modes of the diffusion coefficient.

The approximation space for the solution is determined by the truncation degree \(p_\xi\), hence by the set
\[
\Lambda_{p_\xi}
=
\bigl\{
\bm{\alpha}\in\mathbb N_0^M:\ |\bm{\alpha}|_1\le p_\xi
\bigr\},
\]
whereas the coefficient expansion involves the set \(\mathcal M_a\) in \eqref{eq:coeff_expansion}. These two sets play different roles: \(\Lambda_{p_\xi}\) determines the unknown stochastic degrees of freedom of the discrete solution, while \(\mathcal M_a\) determines which stochastic coupling matrices appear in the operator.

In implementations, one uses the chosen solution truncation \(\Lambda_{p_\xi}\) for the unknowns, while \(\mathcal M_a\) is either known exactly from the coefficient model or obtained from a separate truncation of the input field. The operator assembly then requires only those modes \(\bm{\mu}\in\mathcal M_a\) for which \(a_{\bm{\mu}}\neq 0\).
\end{remark}

\paragraph{Coefficient expansion of the discrete solution.}
On each slab \(I_n\), the discrete solution is written as
\begin{equation}
{u_{h,\tau,p_\xi}}\big|_{I_n}(\bm{x},t,\bm{\xi})
=
\sum_{i=0}^{r}
\sum_{\bm{\alpha}\in\Lambda_{p_\xi}}
\sum_{m=1}^{N_h}
U_{i,\bm{\alpha},m}^{(n)}
\,v_i^{(n)}(t)\,\phi_m(\bm x)\,\Psi_{\bm{\alpha}}(\bm\xi).
\label{eq:discrete_expansion}
\end{equation}
We collect the coefficients \(U_{i,\bm{\alpha},m}^{(n)}\) into a vector \(U^{(n)}\in\mathbb R^{(r+1)N_\xi N_h}, \ \ N_\xi:=\dim \mathcal P_{p_\xi}, \)
ordered by the triple index \((i,\bm{\alpha},m)\).
\paragraph{Temporal matrices on \(I_n\).}
The time discretization contributes the following matrices. The matrix
\begin{align}
[A_t^{(n)}]_{ij}
&:=
\int_{I_n} \partial_t v_j^{(n)}(t)\,v_i^{(n)}(t)\,\mathrm dt
+
v_j^{(n)}(t_{n-1}^+)\,v_i^{(n)}(t_{n-1}^+)
\label{eq:At}
\end{align}
is the \emph{local dG time-evolution matrix}; it contains both the time-derivative term and the current-slab interface contribution.

The matrix
\begin{align}
[B_t^{(n)}]_{ij}
&:=
\int_{I_n} v_j^{(n)}(t)\,v_i^{(n)}(t)\,\mathrm dt
\label{eq:Bt}
\end{align}
is the \emph{local temporal mass matrix}; it weights the spatial stiffness contributions over the slab.

For \(n\ge 2\), the matrix
\begin{align}
[C_t^{(n)}]_{ij}
&:=
v_j^{(n-1)}(t_{n-1}^-)\,v_i^{(n)}(t_{n-1}^+)
\label{eq:Ct}
\end{align}
is the \emph{inter-slab transfer matrix}; it transports the outgoing trace from slab \(I_{n-1}\) to the incoming trace on slab \(I_n\).

\paragraph{Stochastic matrices.}
The stochastic Galerkin projection yields the matrices
\begin{align}
[M_{\bm{\xi}}]_{\bm{\alpha}\bm{\beta}}
:=
\int_\Gamma
\Psi_{\bm{\alpha}}(\bm\xi)\Psi_{\bm{\beta}}(\bm\xi)\rho(\bm\xi)\,\mathrm d\bm\xi
=
\delta_{\bm{\alpha}\bm{\beta}} \ ,
\qquad
[G_{\bm{\mu}}]_{\bm{\alpha}\bm{\beta}} :=
\int_\Gamma
\Psi_{\bm{\mu}}(\bm\xi)\Psi_{\bm{\alpha}}(\bm\xi)\Psi_{\bm{\beta}}(\bm\xi)
\rho(\bm\xi)\,\mathrm d\bm\xi.
\label{eq:triple_stoch}
\end{align}
Here \(M_{\bm{\xi}}\) is the \emph{stochastic mass matrix}; due to the orthonormality of the Hermite basis, it is the identity matrix. The matrices \(G_{\bm{\mu}}\) are the \emph{stochastic triple-product coupling matrices}; they encode how the stochastic modes of the diffusion coefficient couple the stochastic modes of the discrete solution.

\begin{remark}[Factorization of stochastic couplings]
The independence and tensor-product structure also imply a factorization of the stochastic coupling coefficients. For multi-indices \(\bm{\mu},\bm{\alpha},\bm{\beta}\in\mathbb N_0^M\), define
\[
[G_{\bm{\mu}}]_{\bm{\alpha}\bm{\beta}}
=
\int_\Gamma
\Psi_{\bm{\mu}}(\bm\xi)\Psi_{\bm{\alpha}}(\bm\xi)\Psi_{\bm{\beta}}(\bm\xi)\,
\rho(\bm\xi)\,\mathrm d\bm\xi
=
\mathbb E_\rho\!\left[
\Psi_{\bm{\mu}}(\bm\xi)\Psi_{\bm{\alpha}}(\bm\xi)\Psi_{\bm{\beta}}(\bm\xi)
\right].
\]
Since the variables are i.i.d. and the basis is tensorized, one obtains
\[
[G_{\bm{\mu}}]_{\bm{\alpha}\bm{\beta}}
=
\prod_{m=1}^M
\mathbb E\!\left[
\psi_{\mu_m}(\xi_m)\psi_{\alpha_m}(\xi_m)\psi_{\beta_m}(\xi_m)
\right].
\]
This factorization is fundamental for the efficient construction of the stochastic Galerkin matrices.
\end{remark}

\paragraph{Spatial matrices.}
The spatial finite element discretization contributes
\begin{align}
[M_{\bm{x}}]_{m\ell} :=
\left\langle \phi_\ell,\phi_m\right\rangle_{H} \ , \qquad 
[K_{\bm{\mu}}]_{m\ell} :=
\int_D
a_{\bm{\mu}}(\bm x)\,\nabla\phi_\ell(\bm x)\cdot\nabla\phi_m(\bm x)\,\mathrm d\bm x.
\end{align}
Thus \(M_{\bm{x}}\) is the \emph{spatial mass matrix}, while \(K_{\bm{\mu}}\) is the \emph{coefficient-weighted spatial stiffness matrix} associated with the stochastic coefficient mode \(a_{\bm{\mu}}(\bm x)\).

\paragraph{Load vector on \(I_n\).}
The slabwise right-hand side is represented by the vector
\begin{equation}
[F^{(n)}]_{(i,\bm{\alpha},m)}
:=
\int_{I_n}\int_\Gamma
\left\langle
f(\bm x,t,\bm\xi),
\,v_i^{(n)}(t)\Psi_{\bm{\alpha}}(\bm\xi)\phi_m(\bm x)
\right\rangle_{H}
\,\rho(\bm\xi)\,\mathrm d\bm\xi\,\mathrm dt .
\label{eq:load_vector}
\end{equation}
This is the \emph{slab load vector}; it contains the forcing contribution projected onto the full tensor-product test basis.

\paragraph{Initial-value vector on the first slab.}
On the first slab, the incoming trace is given by the projected initial value. This yields the vector
\begin{equation}
[B_0^{(1)}]_{(i,\bm{\alpha},m)}
:=
v_i^{(1)}(t_0^+)\,
\left\langle
u_{0,h,p_\xi,\bm{\alpha}},\phi_m
\right\rangle_H,
\label{eq:initial_vector}
\end{equation}
where \(u_{0,h,p_\xi,\bm{\alpha}}\) denotes the spatial coefficient corresponding to the stochastic mode \(\Psi_{\bm{\alpha}}\) of the projected initial datum \(u_{0,h,p_\xi}\). Hence \(B_0^{(1)}\) is the \emph{initial trace vector} on the first slab.

\begin{problem}[Tensor-product matrix form on each slab]
\label{prob:tensor-product-slab-system}
With these definitions, insertion of \eqref{eq:discrete_expansion} into the
slabwise dG stochastic Galerkin problem shows that the local coefficient vector
\(U^{(n)}\) satisfies
\begin{equation}
\mathbb A^{(n)} U^{(n)} = F^{(n)} + \mathbb J^{(n)} U^{(n-1)},
\qquad n=2,\dots,N,
\label{eq:local_tensor_system}
\end{equation}
where
\begin{equation}
\mathbb A^{(n)}
=
M_{\bm{\xi}}\otimes A_t^{(n)}\otimes M_{\bm{x}}
+
\sum_{\bm{\mu}\in \mathcal M_a}
G_{\bm{\mu}}\otimes B_t^{(n)}\otimes K_{\bm{\mu}},
\ \text{ and } \ 
\mathbb J^{(n)}
=
M_{\bm{\xi}}\otimes C_t^{(n)}\otimes M_{\bm{x}}.
\label{eq:local_rhs_tensor}
\end{equation}
On the first slab one has $\mathbb A^{(1)}U^{(1)}=F^{(1)}+B_0^{(1)}.$
\end{problem}

The interpretation of these terms is as follows. The matrix \(\mathbb A^{(n)}\) is the \emph{local slab operator}; it contains the time-evolution part and the coefficient-weighted diffusion part. The matrix \(\mathbb J^{(n)}\) is the \emph{slab transfer operator}; it injects the outgoing trace from the previous slab into the current slab through the dG jump term.

\begin{remark}[Algebraic solver and matrix-free implementation]
\label{rem:solver-matrix-free-sg} To solve the slab system \eqref{eq:local_tensor_system}, we use the space-time multigrid solver technology developed for tensor-product space-time finite element discretizations of parabolic and hyperbolic problems, Stokes problems, and Navier--Stokes problems; see \cite{NP_parabolic,NMP_hp,MargenbergBause2026NS}. The implementation is based on the matrix-free framework of the deal.II library \cite{dealII97,kronbichlerKormann2012}. Matrix-free means basically that the full global space-time stochastic Galerkin matrix \eqref{eq:local_tensor_system} is not assembled explicitly. Instead, only its action on a block vector is implemented. The small temporal matrices, the sparse stochastic coupling patterns generated by the triple-product matrices \(G_{\bm\mu}\), and local patch matrices used inside the smoother are stored or assembled where needed. The linear systems are solved by flexible GMRES. For the deterministic space-time subsystems, one step of a geometric multigrid \(V\)-cycle is used as preconditioner. The multigrid hierarchy may coarsen in space and, depending on the configuration, also in polynomial degree or time. The smoother is a local space-time Vanka-type patch smoother, where local dense patch problems are formed on space-time patches and inverted approximately. This follows the solver strategy analyzed and tested in \cite{NP_parabolic,NMP_hp,MargenbergBause2026NS}. For the stochastic Galerkin system \eqref{eq:local_tensor_system}, this deterministic space-time solver is embedded into a block-Jacobi preconditioner in the stochastic space. More precisely, the stochastic off-diagonal couplings induced by the matrices \(G_{\bm\mu}\) are kept in the matrix-vector product and are handled by FGMRES, whereas the stochastic diagonal blocks are preconditioned by the deterministic GMG space-time solver. Thus the preconditioner has the form where \(P_{\bm\alpha}^{-1}\) approximates the inverse of the deterministic
space-time block
\[
A_{\bm\alpha}^{(n)}
=
A_t^{(n)}\otimes M_{\bm x}
+
\sum_{\bm\mu\in\mathcal M_a}
[G_{\bm\mu}]_{\bm\alpha\bm\alpha}
B_t^{(n)}\otimes K_{\bm\mu}.
\]
In the present paper, this solver is used as the computational engine for both the stochastic Galerkin runs and the deterministic pathwise solves inside the Monte-Carlo method. A detailed discussion of the stochastic Galerkin multigrid preconditioner, including its implementation and performance analysis, is deferred to the follow-up work \cite{DaworBauseMargenbergSGPreconditioner}.
\end{remark}

\section{Stochastic benchmarks and experimental setup}
\label{sec:manufactured-benchmarks}

Following standard verification and validation principles for prescribed analytical solutions \cite{Roache1998VV,OberkampfTrucano2002}, we validate the intrusive stochastic Galerkin framework using \emph{finite} stochastic polynomial order. It is designed to verify exact stochastic mode capture: once the discrete chaos space contains all exact stochastic modes, the stochastic truncation error must vanish identically, so that only the deterministic space-time discretization error and the algebraic solver tolerance remain.

Throughout this section we use
\[
D=(-1,1)^2,
\qquad
I=(0,T),
\qquad
T=1,
\qquad
M=4,
\qquad
\xi_i \sim \mathcal{N}(0,1) \ \text{ as in } \ (\ref{eq:random_varibale_law})
\]
All stochastic Galerkin results reported below are obtained with the corrected physical right-hand side assembly. The projected forcing modes \(f_{\bm\beta}=\mathbb E_\rho[f\Psi_{\bm\beta}]\) are evaluated analytically from Hermite expansions (projections on stochastic modes with respect to the Hermite basis). A tensor-product Gauss--Hermite projector was used separately to verify these coefficients, and the projected coefficients agreed with the analytic ones to roundoff accuracy.
\subsection{Quantities of interest}
\label{subsec:qoi}

The stochastic Galerkin solution contains more information than a deterministic solution. In particular, we are interested not only in the full stochastic field, but also in its mean behavior and in its fluctuations around the mean. For that reason, the primary quantities of interest are the stochastic mean and the stochastic variance. Let \(\operatorname{He}_n\) denote the probabilists' Hermite polynomial, defined by (\ref{eq:Hermit_def}). The following expectation identities are used repeatedly when interpreting the mean and variance.

\begin{lemma}[Expectation of Hermite polynomials under the Gaussian measure]
\label{lem:expectation-hermite-1d}
Let \(X\sim\mathcal N(0,1)\). Then
\[
\mathbb E_{\rho}[\operatorname{He}_n(X)]
= \int_{\mathbb R}\operatorname{He}_n(x)\rho(x)\,dx
=
\delta_{n0} =
\begin{cases}
1, & n=0,\\[0.3em]
0, & n\ge 1
\end{cases} \ \ \text{ , where } \ \rho(x)=\frac{1}{\sqrt{2\pi}}e^{-x^2/2}
\]
\end{lemma}

\begin{proof}
Use the Rodrigues formula and the one-dimensional Gaussian density.
\end{proof}

\begin{lemma}[Expectation of tensorized multivariate Hermite polynomials]
\label{lem:expectation-multivariate-hermite}
Let \(\bm{\xi}=(\xi_1,\ldots,\xi_M)\) be a standard Gaussian random vector with
independent components. For a multi-index
\(\bm{\alpha}=(\alpha_1,\ldots,\alpha_M)\), define
\[
\operatorname{He}_{\bm{\alpha}}(\bm{\xi})
:=
\prod_{m=1}^{M}
\operatorname{He}_{\alpha_m}(\xi_m).
\]
Then
\[
\mathbb E_{\rho}
\left[
\operatorname{He}_{\bm{\alpha}}(\bm{\xi})
\right]
= \int_{\Gamma}
\operatorname{He}_{\bm{\alpha}}(\bm{\xi})
\rho(\bm{\xi})\,d\bm{\xi}
=
\delta_{\bm{\alpha}\bm{0}} =
\begin{cases}
1, & \bm{\alpha}=\bm{0},\\[0.3em]
0, & \bm{\alpha}\ne\bm{0}.
\end{cases}
\]
\end{lemma}

\begin{proof}
By independence and tensorization,
\[
\mathbb E_\rho[\operatorname{He}_{\bm{\alpha}}(\bm{\xi})]
=
\prod_{m=1}^{M}
\mathbb E[\operatorname{He}_{\alpha_m}(\xi_m)].
\]
Lemma~\ref{lem:expectation-hermite-1d} shows that each factor equals \(1\) if
\(\alpha_m=0\) and vanishes otherwise; hence the product is nonzero only for
\(\bm{\alpha}=\bm{0}\). The normalized identity follows immediately from
\(\Psi_{\bm{\alpha}}
=
\operatorname{He}_{\bm{\alpha}}/\sqrt{\bm{\alpha}!}\).
\end{proof}

For a stochastic field \(w=w(t,\bm{x},\bm{\xi})\), we write the Hermite chaos expansion
\begin{equation}
w(t,\bm{x},\bm{\xi})
=
\sum_{\bm{\alpha}\in\mathbb N_0^M}
w_{\bm{\alpha}}(\bm{x}, t)
\Psi_{\bm{\alpha}}(\bm{\xi}).
\label{eq:general-chaos-expansion}
\end{equation}

\paragraph{Mean.} The stochastic mean is defined by
\begin{equation}
\mathcal{M}[w](\bm{x}, t)
:=
\mathbb E_{\rho}[w](\bm{x}, t)
=
\int_{\Gamma}
w(t,\bm{x},\bm{\xi})\rho(\bm{\xi})\,d\bm{\xi}.
\label{eq:qoi-mean-definition}
\end{equation}
Using \eqref{eq:general-chaos-expansion} and Lemma~\ref{lem:expectation-multivariate-hermite}, we obtain
\begin{equation}
\mathcal{M}[w](\bm{x}, t)
=
w_{\bm{0}}(\bm{x}, t).
\label{eq:qoi-mean-chaos}
\end{equation}
Thus, the mean is exactly the zeroth stochastic Galerkin mode.

\paragraph{Variance.}
The stochastic variance is defined pointwise in \((\bm{x}, t)\) by
\begin{equation}
\mathcal{V}[w](\bm{x}, t)
:=
\operatorname{Var}_{\rho}(w)(\bm{x}, t)
=
\mathbb E_{\rho}
\left[
\left(
w(t,\bm{x},\bm{\xi})
-
\mathcal{M}[w](\bm{x}, t)
\right)^2
\right].
\label{eq:qoi-variance-definition}
\end{equation}
By the orthonormality relation \eqref{eq:hermite-orthonormality}, this reduces
to
\begin{equation}
\mathcal{V}[w](\bm{x}, t)
=
\sum_{\bm{\alpha}\ne\bm{0}}
\left|
w_{\bm{\alpha}}(\bm{x}, t)
\right|^2.
\label{eq:qoi-variance-chaos}
\end{equation}
The variance therefore measures the contribution of all nonzero stochastic modes. This is why the variance is especially sensitive to stochastic refinement.

\subsection{Error norms}
\label{subsec:error-norms}

We now define the norms used to measure the error in the quantities of interest from Subsection~\ref{subsec:qoi}. These definitions are used later in all result tables.

\begin{definition}[Error norms and error quantities]
\label{def:error-norms-and-quantities}
For a deterministic field \(z=z(\bm{x}, t)\), we use
\begin{equation}
\|z\|_{L^2_t(L^2_x)}
:=
\left(
\int_0^T
\|z(t,\cdot)\|_{L^2(D)}^2\,dt
\right)^{1/2},
\label{eq:deterministic-L2t-L2x-norm}
\end{equation}
and the corresponding space-time \(H^1\)-seminorm
\begin{equation}
\|z\|_{L^2_t(H^1_x\text{-semi})}
:=
\left(
\int_0^T
\|\nabla_{\bm{x}}z(t,\cdot)\|_{L^2(D)}^2\,dt
\right)^{1/2}.
\label{eq:deterministic-L2t-H1semi-norm}
\end{equation}

For a stochastic field \(w=w(t,\bm{x},\bm{\xi})\), the full stochastic space-time \(L^2\)-norm is
\begin{equation}
\|w\|_{L^2_t(L^2_{\rho}(L^2_x))}
:=
\left(
\int_0^T
\int_{\Gamma}
\|w(t,\cdot,\bm{\xi})\|_{L^2(D)}^2
\rho(\bm{\xi})\,d\bm{\xi}\,dt
\right)^{1/2}.
\label{eq:full-stochastic-L2-norm}
\end{equation}
The corresponding full stochastic \(H^1\)-seminorm is
\begin{equation}
\|w\|_{L^2_t(L^2_{\rho}(H^1_x\text{-semi}))}
:=
\left(
\int_0^T
\int_{\Gamma}
\|\nabla_{\bm{x}}w(t,\cdot,\bm{\xi})\|_{L^2(D)}^2
\rho(\bm{\xi})\,d\bm{\xi}\,dt
\right)^{1/2}.
\label{eq:full-stochastic-H1semi-norm}
\end{equation}

By \eqref{eq:hermite-orthonormality}, the full stochastic norms can be written as modal sums:
\begin{align}
\|w\|_{L^2_t(L^2_{\rho}(L^2_x))}^2
&=
\sum_{\bm{\alpha}}
\|w_{\bm{\alpha}}\|_{L^2_t(L^2_x)}^2,
\label{eq:modal-sum-L2}
\\
\|w\|_{L^2_t(L^2_{\rho}(H^1_x\text{-semi}))}^2
&=
\sum_{\bm{\alpha}}
\|w_{\bm{\alpha}}\|_{L^2_t(H^1_x\text{-semi})}^2.
\label{eq:modal-sum-H1semi}
\end{align}

Let \( e_{h,\tau,p_\xi} := u-u_{h,\tau,p_\xi} \) be the total approximation error. The full stochastic errors used in the experiments are
\begin{align}
E_{p_\xi}^{\mathrm{full},L^2}
&:=
\|e_{h,\tau,p_\xi}\|_{L^2_t(L^2_{\rho}(L^2_x))},
\label{eq:error-full-L2}
\\
E_{p_\xi}^{\mathrm{full},H^1}
&:=
\|e_{h,\tau,p_\xi}\|_{L^2_t(L^2_{\rho}(H^1_x\text{-semi}))}.
\label{eq:error-full-H1}
\end{align}
The mean errors are
\begin{align}
E_{p_\xi}^{\mathrm{mean},L^2}
&:=
\|\mathcal{M}[u]-\mathcal{M}[u_{h,\tau,p_\xi}]\|_{L^2_t(L^2_x)},
\label{eq:error-mean-L2}
\\
E_{p_\xi}^{\mathrm{mean},H^1}
&:=
\|\mathcal{M}[u]-\mathcal{M}[u_{h,\tau,p_\xi}]\|_{L^2_t(H^1_x\text{-semi})}.
\label{eq:error-mean-H1}
\end{align}
The variance errors are
\begin{align}
E_{p_\xi}^{\mathrm{var},L^2}
&:=
\|\mathcal{V}[u]-\mathcal{V}[u_{h,\tau,p_\xi}]\|_{L^2_t(L^2_x)},
\label{eq:error-var-L2}
\\
E_{p_\xi}^{\mathrm{var},H^1}
&:=
\|\mathcal{V}[u]-\mathcal{V}[u_{h,\tau,p_\xi}]\|_{L^2_t(H^1_x\text{-semi})}.
\label{eq:error-var-H1}
\end{align}
\end{definition}

The full errors \eqref{eq:error-full-L2}--\eqref{eq:error-full-H1} measure the complete stochastic approximation error. The mean errors \eqref{eq:error-mean-L2}--\eqref{eq:error-mean-H1} measure the error in the deterministic average field. The variance errors \eqref{eq:error-var-L2}--\eqref{eq:error-var-H1} measure the error in the stochastic fluctuations.

\subsection{Finite polynomial analytical solution}
\label{subsec:finite-polynomial-benchmark}

We now test the method with a prescribed solution of finite stochastic polynomial degree. The prescribed solution uses the following diffusion with random variables
\begin{equation}
a(\bm{x},\bm{\xi})
=
d_{\min}
+
\bigl(
\xi_1x_1^2
+
\xi_2x_1x_2
+
\xi_3x_2^2
+
\xi_4
\bigr)^2,
\qquad
\bm{x}=(x_1,x_2)\in D,
\quad
\bm{\xi}\in\Gamma=\mathbb{R}^4.
\label{eq:diffusion_benchmark}
\end{equation}
Here \(d_{\min}>0\) is fixed. Therefore
\begin{equation}
a(\bm{x},\bm{\xi})
\ge d_{\min}>0
\qquad
\text{for all }(\bm{x},\bm{\xi})\in D\times\Gamma.
\label{eq:diffusion-lower-bound}
\end{equation}
In our computations we have choosen $d_{\min} = 0.2$, however any choice of $d_{\min}$ greater than zero would be sufficient. 

\begin{remark}[Positivity versus global boundedness]
The coefficient \eqref{eq:diffusion_benchmark} is uniformly positive, but it is not uniformly bounded from above on all of \(\Gamma\), because the Gaussian variables \(\xi_1,\ldots,\xi_4\) are unbounded. More precisely, for each fixed
\(i\),
\[
\lim_{A\to\infty}\mathbb{P}(\xi_i\ge A)=0,
\qquad
\lim_{A\to\infty}\mathbb{P}(|\xi_i|\ge A)=0,
\]
but there is no finite deterministic constant \(a_{\max}\) such that
\[
a(\bm{x},\bm{\xi})
\le a_{\max}
\qquad
\text{for all }(\bm{x},\bm{\xi})\in D\times\Gamma.
\]
Nevertheless, this lack of a global upper bound is caused only by rare large
Gaussian realizations. On
\[
\Gamma_A
:=
\left\{
\bm{\xi}\in\mathbb{R}^4:
\max_{m=1,\ldots,4}|\xi_m|\le A
\right\},
\]
we have
\[
\left|
\xi_1x_1^2+\xi_2x_1x_2+\xi_3x_2^2+\xi_4
\right|
\le4A
\qquad
\text{for all }\bm{x}\in D,
\]
because \(|x_1|\le1\), \(|x_2|\le1\), and \(|x_1x_2|\le1\). Hence
\[
a(\bm{x},\bm{\xi})
\le d_{\min}+16A^2
\qquad
\text{for all }(\bm{x},\bm{\xi})\in D\times\Gamma_A.
\]
Moreover,
\[
\mathbb{P}(\Gamma_A)\to1
\qquad
\text{as }A\to\infty.
\]
Thus the coefficient is uniformly bounded on events of arbitrarily large probability, which is sufficient for the practical numerical interpretation of this benchmark.
\end{remark}

The analytical exact solution is chosen in separated form,
\begin{equation}
u_{\mathrm{ex}}(t,\bm{x},\bm{\xi})
=
\phi(\bm{x}, t)G_q(\bm{\xi}),
\label{eq:finite-exact-separated}
\end{equation}
with deterministic space-time factor
\begin{equation}
\phi(\bm{x}, t)
=
t^4
\sin\!\left(\frac{\pi}{2}(x_1+1)\right)
\sin\!\left(\frac{\pi}{2}(x_2+1)\right).
\label{eq:deterministic-factor-phi}
\end{equation}
The stochastic factor is a finite Hermite polynomial chaos expansion,
\begin{equation}
G_q(\bm{\xi})
=
\sum_{\bm{\alpha}\in\Lambda_q}
c_{\bm{\alpha}}\Psi_{\bm{\alpha}}(\bm{\xi}),
\qquad
\Lambda_q
=
\left\{
\bm{\alpha}\in\mathbb{N}_0^M:
|\bm{\alpha}|_1\le q
\right\}.
\label{eq:finite-hermite-factor}
\end{equation}
and the coefficients are chosen as
\begin{equation}
c_{\bm{\alpha}}
=
\begin{cases}
1, & \bm{\alpha}=\bm{0},\\[1mm]
\eta^{|\bm{\alpha}|_1}, & \bm{\alpha}\ne\bm{0},
\end{cases}
\qquad
\label{eq:finite-chaos-coefficients}
\end{equation}
where $\eta \in \mathbb{R}_{+}$.
In the finite-chaos experiment we use $M=4$, $q=6$. Therefore, the exact analytical solution is explicitly
\begin{equation}
\boxed{
u_{\mathrm{ex}}(t,x_1,x_2,\xi_1,\xi_2,\xi_3,\xi_4)
=
t^4
\sin\!\left(\frac{\pi}{2}(x_1+1)\right)
\sin\!\left(\frac{\pi}{2}(x_2+1)\right)
\sum_{\alpha_1+\alpha_2+\alpha_3+\alpha_4\le6}
c_{\bm{\alpha}}
\prod_{m=1}^{4}
\frac{\operatorname{He}_{\alpha_m}(\xi_m)}
{\sqrt{\alpha_m!}}
}
\label{eq:finite-exact-solution-M4-q6}
\end{equation}
\begin{remark}[Finite-chaos modal truncation]
The corresponding modal coefficients are
\begin{equation}
u_{\bm{\alpha}}(\bm{x}, t)
=
c_{\bm{\alpha}}\phi(\bm{x}, t),
\qquad
\bm{\alpha}\in\Lambda_q,
\label{eq:finite-modal-coefficients-active}
\end{equation}
and
\begin{equation}
u_{\bm{\alpha}}(\bm{x}, t)=0,
\qquad
\bm{\alpha}\notin\Lambda_q.
\label{eq:finite-modal-coefficients-inactive}
\end{equation}

This immediately explains what we expect numerically. If \(p_\xi<q\), then
the stochastic Galerkin approximation omits exact modes in
\(\Lambda_q\setminus\Lambda_{p_\xi}\). By the modal identity
\eqref{eq:modal-sum-L2}, the pure stochastic truncation error is
\begin{equation}
\left(E_{p_\xi}^{\mathrm{trunc},L^2}\right)^2
=
\|\phi\|_{L^2_t(L^2_x)}^2
\sum_{\bm{\alpha}\in\Lambda_q\setminus\Lambda_{p_\xi}}
|c_{\bm{\alpha}}|^2.
\label{eq:finite-truncation-error-L2}
\end{equation}
Once \(p_\xi\ge q\), all exact stochastic modes are contained in the discrete
chaos space, and hence
\begin{equation}
E_{p_\xi}^{\mathrm{trunc},L^2}=0.
\label{eq:finite-truncation-zero}
\end{equation}
Thus, after \(p_\xi=q=6\), any remaining error is caused by the fixed
deterministic space-time discretization and the algebraic solver tolerance,
not by stochastic truncation.
\end{remark}

The solution satisfies the homogeneous initial condition,
\[
u_{\mathrm{ex}}(0,\bm{x},\bm{\xi})=0,
\]
because of the factor \(t^4\). It also satisfies homogeneous Dirichlet
boundary conditions on \(\partial D\), since the sine factors in
\eqref{eq:deterministic-factor-phi} vanish for \(x_1=\pm1\) or \(x_2=\pm1\).

The right-hand side is defined by inserting the exact solution into the random
heat equation:
\begin{equation}
f(t,\bm{x},\bm{\xi})
=
\partial_tu_{\mathrm{ex}}(t,\bm{x},\bm{\xi})
-
\nabla_{\bm{x}}\cdot
\left(
a(\bm{x},\bm{\xi})
\nabla_{\bm{x}}u_{\mathrm{ex}}(t,\bm{x},\bm{\xi})
\right).
\label{eq:manufactured_rhs_benchmark}
\end{equation}
Since \(G_q\) is independent of \(t\) and \(\bm{x}\), this becomes
\begin{equation}
f(t,\bm{x},\bm{\xi})
=
\left[
\partial_t\phi(\bm{x}, t)
-
\nabla_{\bm{x}}\cdot
\left(
a(\bm{x},\bm{\xi})
\nabla_{\bm{x}}\phi(\bm{x}, t)
\right)
\right]
G_q(\bm{\xi}).
\label{eq:finite_rhs_factorized}
\end{equation}
Thus the forcing carries the same finite-chaos stochastic factor as the exact
solution itself.

\subsection{Finite polynomial benchmark: numerical results}
\label{subsec:finite-results}
The finite-chaos benchmark is computed at a fixed deterministic space-time resolution in order to isolate the effect of increasing the stochastic polynomial degree \(p_\xi\). The physical domain is \(D=(-1,1)^2\), the final time is \(T=1\), and the random input consists of \(M=4\) independent standard Gaussian variables. The prescribed solution has finite Hermite degree \(q=6\), with coefficients $\eta=0.35$. The deterministic discretization uses conforming finite elements of degree \(p_x=3\) in space and a \(\mathrm{dG}(3)\) time discretization with a right Gauss--Radau nodal basis on each time slab. For the stochastic-refinement study we fix refinement level \(6\), corresponding to \(16384\) cells, \(N_x=148225\) spatial degrees of freedom, time step \(\tau=\Delta t=0.015625\), and \(64\) time slabs. The stochastic polynomial degree is varied as \(p_\xi=0,\ldots,6\), giving \(N_\xi=\binom{M+p_\xi}{p_\xi}\) stochastic degrees of freedom. The resulting slab systems are solved by flexible GMRES with the block-Jacobi GMG preconditioner in stochastic space.\\

The error quantities are shown in Table~\ref{tab:finite-errors-rates}. The table reports the full stochastic errors, the mean errors, and the variance errors from \eqref{eq:error-full-L2}, \eqref{eq:error-full-H1}, \eqref{eq:qoi-mean-chaos}, \eqref{eq:error-var-L2}, and \eqref{eq:error-var-H1}. It also lists empirical rates with respect to the growth of the stochastic space dimension. The corresponding graphical summaries are shown in Figures~\ref{fig:stochastic-refinement-error} and \ref{fig:stochastic-refinement-rate}.

\begin{table}[htbp]
\centering
\caption{Error quantities and empirical stochastic rates with respect to stochastic dimension growth for the finite polynomial stochastic-refinement benchmark.}
\label{tab:finite-errors-rates}

\scriptsize
\setlength{\tabcolsep}{2.6pt}
\renewcommand{\arraystretch}{1.08}

\resizebox{\linewidth}{!}{%
\begin{tabular}{rrrrrrrrrrr}
\toprule
\multicolumn{2}{c}{Stochastic space}
& \multicolumn{6}{c}{Error quantities}
& \multicolumn{3}{c}{Dimension-growth rates} \\
\cmidrule(lr){1-2}
\cmidrule(lr){3-8}
\cmidrule(lr){9-11}
\(p_\xi\)
& \(N_\xi\)
& \(E_{p_\xi}^{\mathrm{full},L^2}\)
& \(E_{p_\xi}^{\mathrm{full},H^1}\)
& \(E_{p_\xi}^{\mathrm{mean},L^2}\)
& \(E_{p_\xi}^{\mathrm{mean},H^1}\)
& \(E_{p_\xi}^{\mathrm{var},L^2}\)
& \(E_{p_\xi}^{\mathrm{var},H^1}\)
& \(r_\xi(E^{\mathrm{full},L^2})\)
& \(r_\xi(E^{\mathrm{mean},L^2})\)
& \(r_\xi(E^{\mathrm{var},L^2})\) \\
\midrule
0 &   1
& \(2.76327\times 10^{-1}\)
& \(6.13993\times 10^{-1}\)
& \(8.62293\times 10^{-3}\)
& \(2.34782\times 10^{-2}\)
& \(1.24882\times 10^{-1}\)
& \(3.20336\times 10^{-1}\)
& --     & --     & --     \\

1 &   5
& \(1.48166\times 10^{-1}\)
& \(3.29533\times 10^{-1}\)
& \(8.62293\times 10^{-3}\)
& \(2.34782\times 10^{-2}\)
& \(3.14508\times 10^{-2}\)
& \(8.22517\times 10^{-2}\)
& 0.387  & 0.000  & 0.857  \\

2 &  15
& \(7.22494\times 10^{-2}\)
& \(1.60793\times 10^{-1}\)
& \(3.77096\times 10^{-4}\)
& \(1.46579\times 10^{-3}\)
& \(2.97674\times 10^{-3}\)
& \(1.06823\times 10^{-2}\)
& 0.654  & 2.849  & 2.146  \\

3 &  35
& \(3.31301\times 10^{-2}\)
& \(7.37854\times 10^{-2}\)
& \(3.77096\times 10^{-4}\)
& \(1.46579\times 10^{-3}\)
& \(3.06690\times 10^{-4}\)
& \(1.57806\times 10^{-3}\)
& 0.920  & 0.000  & 2.682  \\

4 &  70
& \(1.44300\times 10^{-2}\)
& \(3.21591\times 10^{-2}\)
& \(2.31986\times 10^{-5}\)
& \(1.14981\times 10^{-4}\)
& \(5.41277\times 10^{-5}\)
& \(1.60359\times 10^{-4}\)
& 1.199  & 4.023  & 2.502  \\

5 & 126
& \(5.70228\times 10^{-3}\)
& \(1.27147\times 10^{-2}\)
& \(2.31986\times 10^{-5}\)
& \(1.14981\times 10^{-4}\)
& \(1.45194\times 10^{-5}\)
& \(3.90733\times 10^{-5}\)
& 1.580  & 0.000  & 2.239  \\

6 & 210
& \(4.22465\times 10^{-10}\)
& \(4.47915\times 10^{-8}\)
& \(3.25162\times 10^{-10}\)
& \(3.44901\times 10^{-8}\)
& \(1.15674\times 10^{-10}\)
& \(2.58393\times 10^{-8}\)
& 32.140 & 21.877 & 22.983 \\

7 & 330
& \(4.22119\times 10^{-10}\)
& \(4.47912\times 10^{-8}\)
& \(3.25040\times 10^{-10}\)
& \(3.44901\times 10^{-8}\)
& \(1.15616\times 10^{-10}\)
& \(2.58392\times 10^{-8}\)
& 0.000  & 0.000  & 0.000  \\

8 & 495
& \(4.22119\times 10^{-10}\)
& \(4.47912\times 10^{-8}\)
& \(3.25040\times 10^{-10}\)
& \(3.44901\times 10^{-8}\)
& \(1.15616\times 10^{-10}\)
& \(2.58392\times 10^{-8}\)
& 0.000  & 0.000  & 0.000  \\
\bottomrule
\end{tabular}%
}
\end{table}

The behavior in Table~\ref{tab:finite-errors-rates} and Figure~\ref{fig:stochastic-refinement-error} reflects the finite-chaos structure of the analytical prescribed solution. For \(p_\xi<q=6\), the SG space does not yet contain all modes in \(\Lambda_q\). Consequently, the full stochastic errors and the variance errors decrease as the omitted Hermite tail is reduced. The full \(L^2\)-error decreases from \(2.76327\times10^{-1}\) at \(p_\xi=0\) to \(5.70228\times10^{-3}\) at \(p_\xi=5\), while the variance \(L^2\)-error decreases from \(1.24882\times10^{-1}\) to \(1.45194\times10^{-5}\). At \(p_\xi=q=6\), the finite stochastic expansion is fully represented. The full, mean, and variance errors then drop simultaneously to the deterministic/algebraic floor, in agreement with \eqref{eq:finite-truncation-zero}.\\

The mean-error columns show a characteristic pairwise behavior. The mean is the zeroth chaos mode, as stated in \eqref{eq:qoi-mean-chaos}, but the equation for this mode is not independent of the other stochastic modes. Through the random diffusion coefficient and the triple products \eqref{eq:triple_stoch}, the mean equation is coupled to the retained fluctuation modes. In the present benchmark the diffusion coefficient is a square of a linear Gaussian expression and therefore has an even stochastic structure. Thus, the mean error is unchanged when passing from \(p_\xi=0\) to \(p_\xi=1\), drops at \(p_\xi=2\), remains unchanged at \(p_\xi=3\), drops again at \(p_\xi=4\), and remains unchanged at \(p_\xi=5\). Thus the pattern \((0,1),(2,3),(4,5)\) is not a numerical accident: the odd enrichment levels add modes that do not improve the even stochastic chain coupled to the mean, whereas the even enrichment levels add the next relevant modes. Finally, \(p_\xi=6\) completes the prescribed finite chaos and reduces the mean error to the fixed floor.\\

\begin{figure}[ht]
\centering
\includegraphics[width=0.9\linewidth]{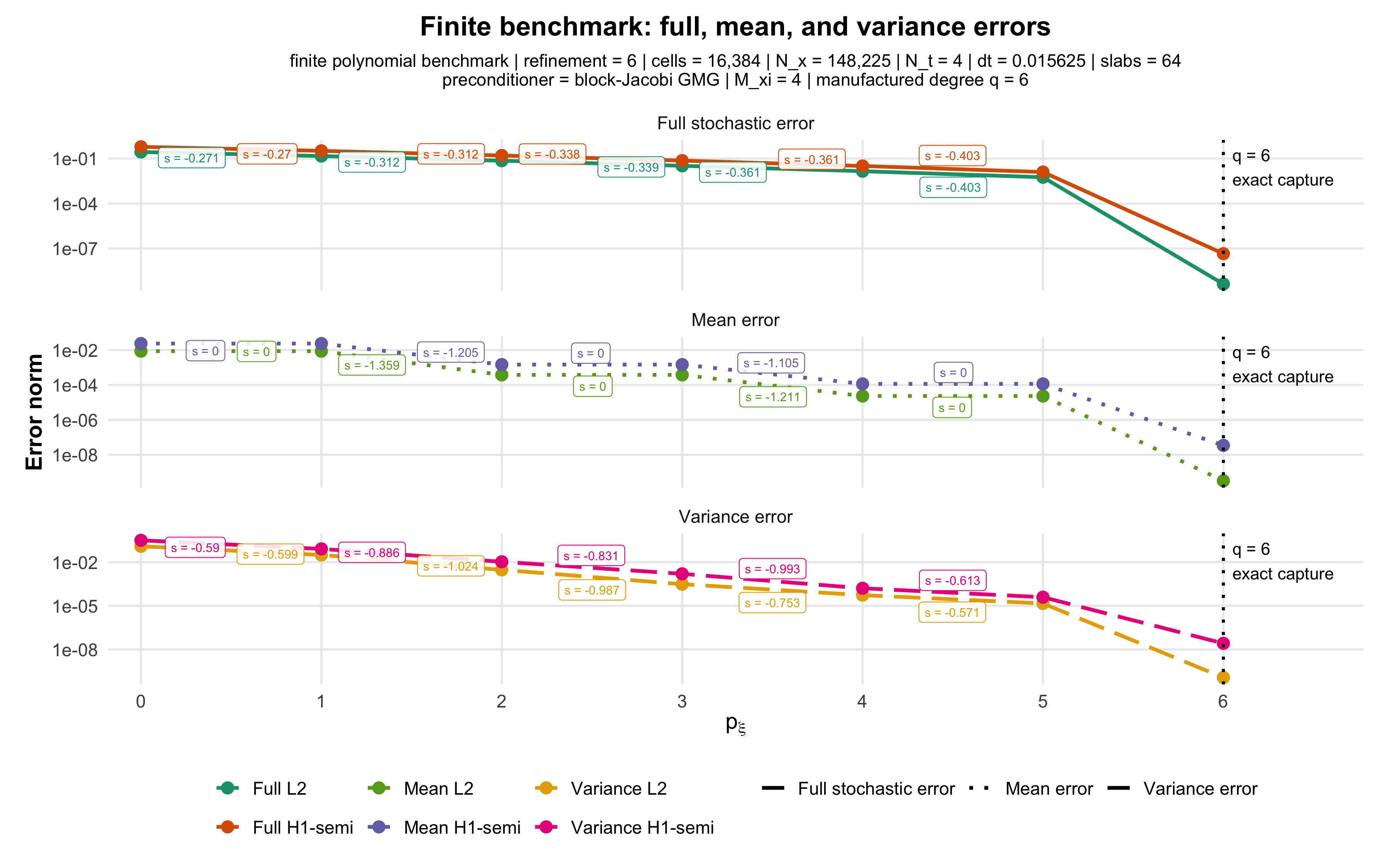}
\caption{Stochastic-refinement error diagnostics.}
\label{fig:stochastic-refinement-error}
\end{figure}

\begin{figure}[ht]
\centering
\includegraphics[width=0.9\linewidth]{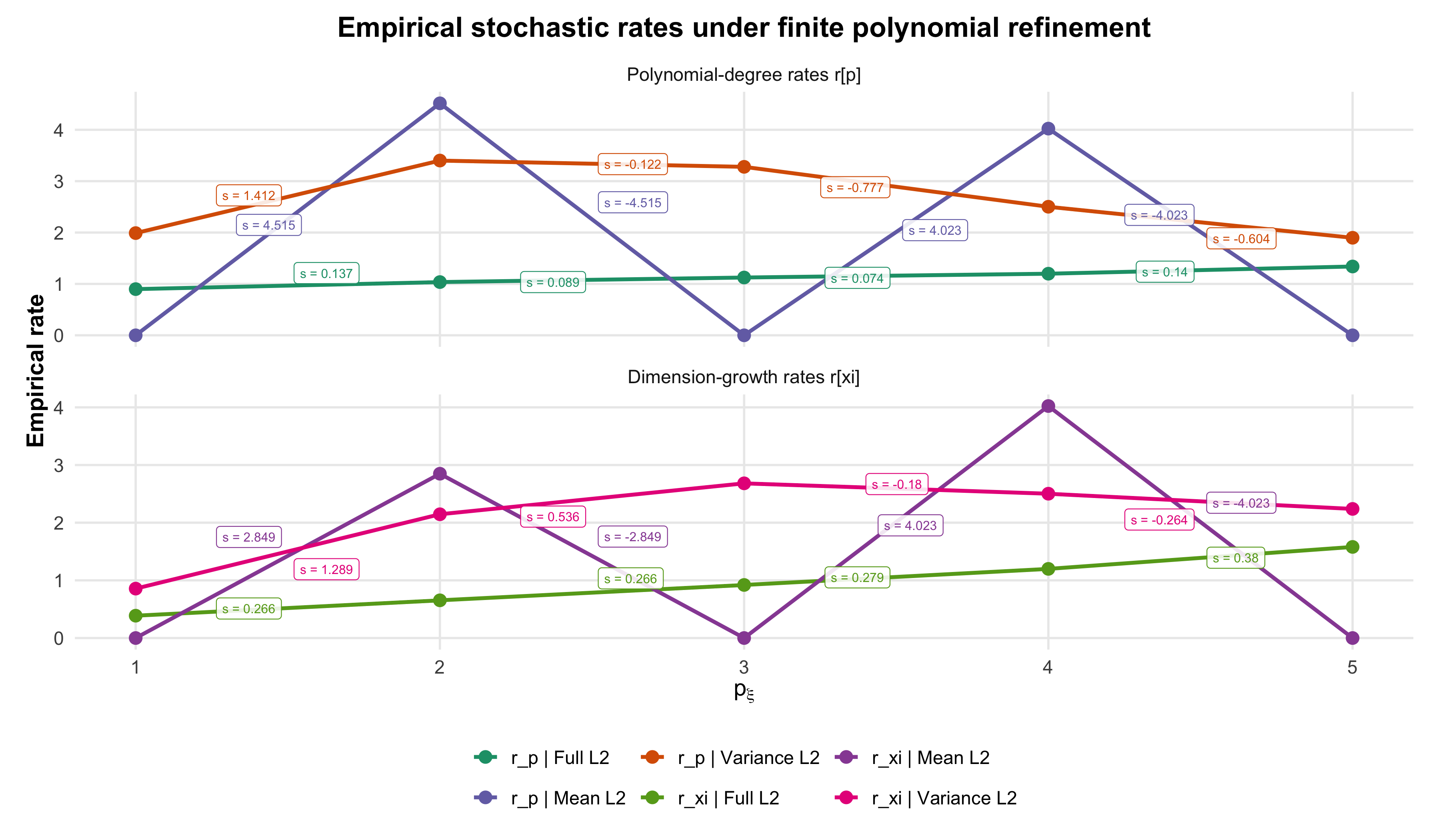}
\caption{Stochastic-refinement rate diagnostics.}
\label{fig:stochastic-refinement-rate}
\end{figure}

It is also useful to measure rates with respect to the growth of the stochastic
dimension \(N_\xi\). For an error sequence \(E_{p_\xi}\), we use
\begin{equation}
 r_p(p_\xi)=\log\!\left(\frac{E_{p_\xi-1}}{E_{p_\xi}}\right),
 \qquad
 r_\xi(E;p_\xi)=
 \frac{\log(E_{p_\xi-1}/E_{p_\xi})}
 {\log(N_\xi(p_\xi)/N_\xi(p_\xi-1))},
 \qquad p_\xi\ge1.
\label{eq:stochastic-dimension-rate}
\end{equation}
The corresponding empirical rates are listed in Table~\ref{tab:finite-errors-rates} and plotted in Figure~\ref{fig:stochastic-refinement-rate} for the pre-floor range \(p_\xi=1,\ldots,5\). The full-error rates increase over this range, which indicates that the stochastic truncation error is reduced more efficiently as the polynomial space approaches the exact degree. The variance rates are larger than the full-error rates, showing that the fluctuation component is resolved particularly efficiently by stochastic enrichment. The mean-rate plot again reflects the even--odd behavior discussed above: the rate is zero when an odd degree is added and large when the next even degree is included. The very large rates reported in Table~\ref{tab:finite-errors-rates} at \(p_\xi=6\) are not asymptotic rates; they record the finite-chaos capture jump. For \(p_\xi>6\), the stochastic truncation has already vanished and the errors remain on the fixed space-time and algebraic floor.

\subsection{Algebraic size, solver performance, and computational work for the finite polynomial benchmark}
\label{subsec:finite-computational-work}

The finite benchmark is not only a convergence test; it also shows the
computational cost of increasing the stochastic polynomial degree. For a
total-degree index set $\Lambda_{p_\xi} = \{\bm{\alpha}\in\mathbb{N}_0^M:|\bm{\alpha}|_1\le p_\xi\},$ the number of stochastic modes is
\begin{equation}
N_\xi(p_\xi)
=
|\Lambda_{p_\xi}|
=
\binom{M+p_\xi}{p_\xi}.
\label{eq:stochastic-dofs-count}
\end{equation}
Since \(M=4\), $N_\xi(p_\xi)=\binom{4+p_\xi}{p_\xi}.$ On each time slab, the stochastic Galerkin size is
\begin{equation}
N_{\mathrm{SG,slab}}(p_\xi)
=
N_\xi(p_\xi)\,N_t\,N_x.
\label{eq:sg-slab-size}
\end{equation}
Here $ N_t=4, \ N_x=148225.$ As a machine-independent work proxy, we use
\begin{equation}
W(p_\xi)
:=
N_{\mathrm{SG,slab}}(p_\xi)
\cdot
\texttt{prec\_calls}(p_\xi).
\label{eq:work-proxy}
\end{equation}
This quantity combines the algebraic size of one stochastic Galerkin slab system with the number of preconditioner applications required by FGMRES. In Table~\ref{tab:finite-size-work-solver-performance}, \(\text{avg.\ FGMRES}\) denotes the average number of FGMRES iterations per slab, computed as the total number of Krylov steps divided by the number of slab batches. The columns ``vmult calls'' and ``prec. calls'' count, respectively, applications of the stochastic Galerkin operator and applications of the SG preconditioner. The time \(\text{wall}\) is the total measured wall-clock time for the refinement run, while \(\text{solve}\) is the accumulated time spent inside the FGMRES solve calls. Finally, \(\text{apply}\) and \(\text{prec}\) are accumulated timings for stochastic operator applications and preconditioner applications, respectively.

\begin{table}
\caption{Algebraic size, work proxy, solver statistics, and timing data for the
finite polynomial stochastic-refinement benchmark.}
\label{tab:finite-size-work-solver-performance}

\scriptsize
\setlength{\tabcolsep}{2.2pt}
\renewcommand{\arraystretch}{1.08}

\resizebox{\linewidth}{!}{%
\begin{tabular}{@{}rrrrrrrrrrr@{}}
\hline
\multicolumn{3}{c}{Algebraic size}
& \multicolumn{3}{c}{Solver statistics}
& \multicolumn{4}{c}{Timing}
& \multicolumn{1}{c}{Work} \\
\cline{1-3}\cline{4-6}\cline{7-10}\cline{11-11}
\(p_\xi\)
& \(N_\xi\)
& \(N_{\mathrm{SG,slab}}\)
& avg. FGMRES
& vmult calls
& prec. calls
& wall [s]
& solve [s]
& apply [s]
& prec. [s]
& \(W\) \\
\hline
0 &   1 &       592900 &  8.7031 &  621 &  557 &     48.7720 &     27.5779 &    0.2910 &     27.2334 & \(3.30245300\times10^{8}\) \\
1 &   5 &      2964500 & 12.6250 &  872 &  808 &    290.2853 &    204.6696 &    5.9563 &    198.6127 & \(2.39531600\times10^{9}\) \\
2 &  15 &      8893500 & 17.7344 & 1199 & 1135 &   1135.3510 &    877.4686 &   37.3328 &    839.0575 & \(1.00941225\times10^{10}\) \\
3 &  35 &     20751500 & 21.0000 & 1408 & 1344 &   3040.5230 &   2454.6130 &  123.2545 &   2327.3950 & \(2.78900160\times10^{10}\) \\
4 &  70 &     41503000 & 24.9531 & 1665 & 1597 &   7070.7110 &   5910.6170 &  328.1399 &   5570.0270 & \(6.62802910\times10^{10}\) \\
5 & 126 &     74705400 & 27.6250 & 1858 & 1768 &  13922.5400 &  11881.4700 &  705.7213 &  11150.1200 & \(1.32079147\times10^{11}\) \\
6 & 210 &    124509000 & 31.0156 & 2090 & 1985 &  25834.0100 &  22485.8000 & 1388.9440 &  21046.5600 & \(2.47150365\times10^{11}\) \\
\hline
\end{tabular}
}
\end{table}

\begin{figure}[htbp]
\centering
\includegraphics[width=0.72\linewidth]{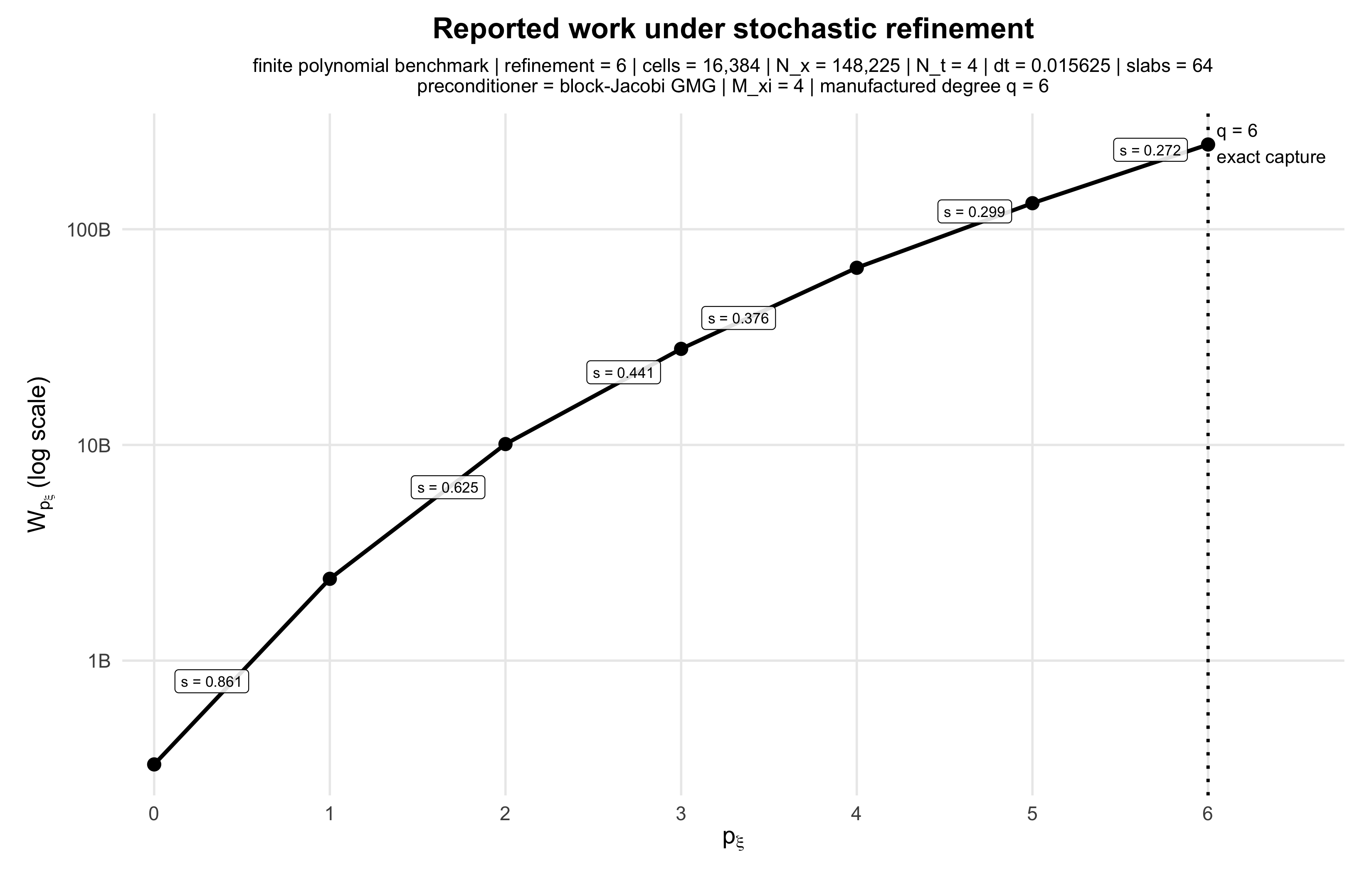}
\caption{Reported work proxy \(W(p_\xi)\) for the finite polynomial stochastic-refinement benchmark.}
\label{fig:work}
\end{figure}

\begin{figure}[tbp]
\centering
\includegraphics[width=0.9\linewidth]{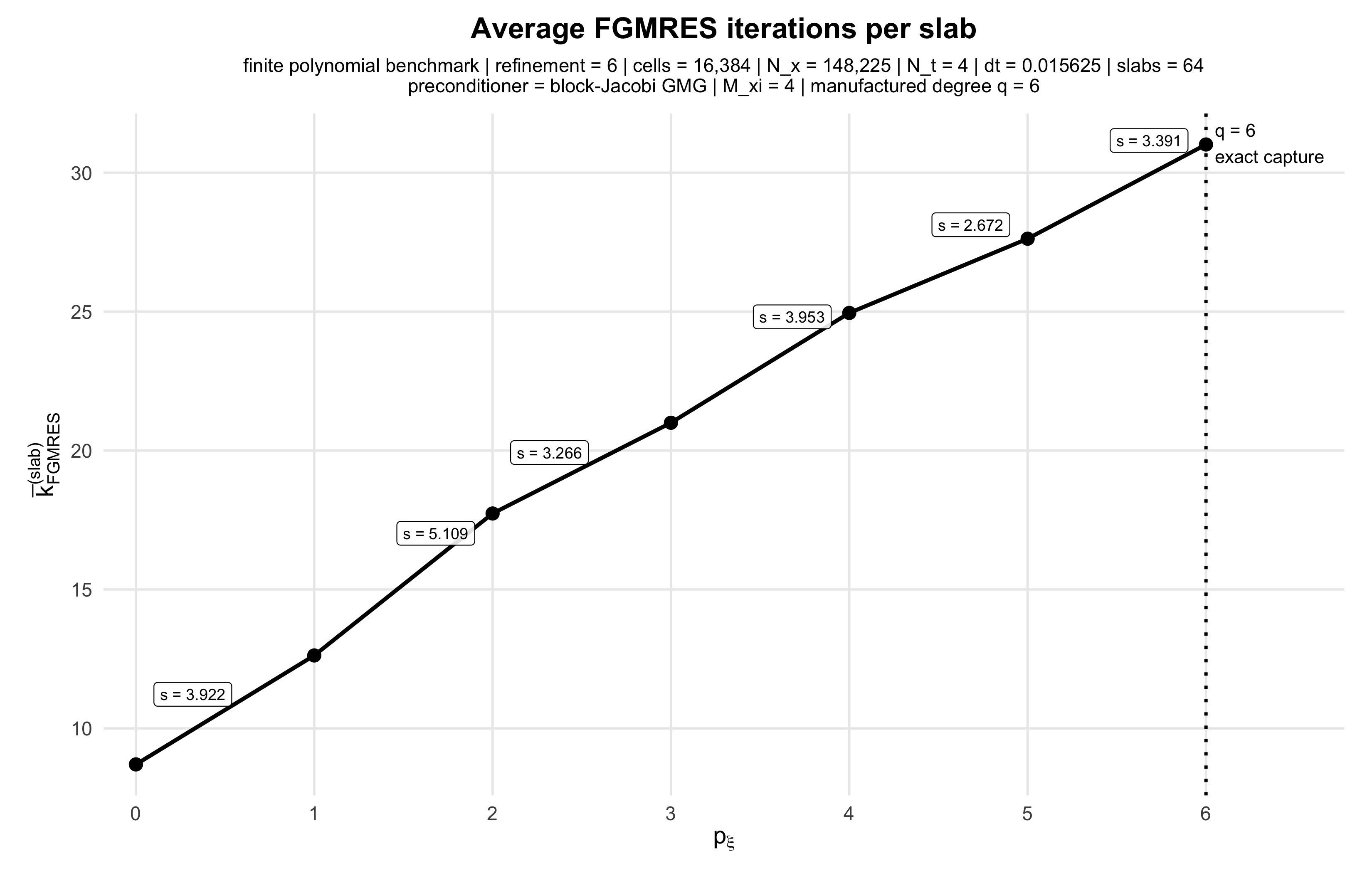}
\caption{Average FGMRES iterations per time slab for the stochastic-refinement study as a function of the chaos degree \(p_\xi\).}
\label{fig:stochastic-refinement-fgmres}
\end{figure}

\begin{figure}[tbp]
\centering
\includegraphics[width=0.9\linewidth]{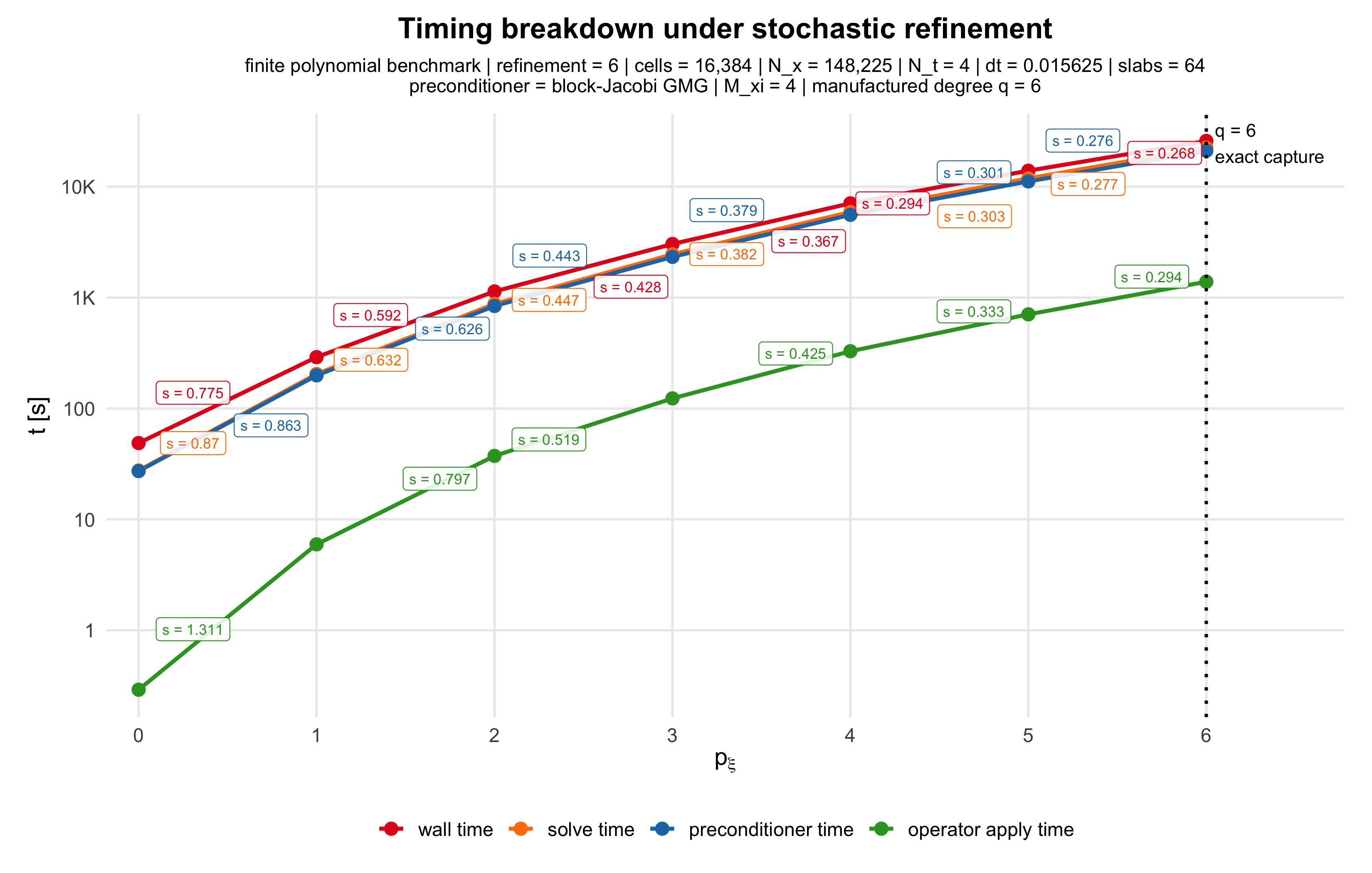}
\caption{Timing breakdown for the stochastic Galerkin solves in the stochastic-refinement study.}
\label{fig:stochastic-refinement-timing-breakdown}
\end{figure}

Table~\ref{tab:finite-size-work-solver-performance} and Figures~\ref{fig:work}--\ref{fig:stochastic-refinement-timing-breakdown} show the computational cost of the same stochastic-refinement experiment. In this corrected run we report the data up to the finite-chaos capture point \(p_\xi=q=6\). The number of stochastic modes increases from \(N_\xi=1\) to \(N_\xi=210\), and the slabwise SG size increases from \(5.92900\times10^5\) to \(1.24509\times10^8\) degrees of freedom. The average FGMRES iteration count rises from \(8.7031\) to \(31.0156\) iterations per slab. Thus stochastic enrichment increases both the algebraic dimension and the Krylov difficulty of the slab solve.

The work proxy \eqref{eq:work-proxy} grows from \(3.30245300\times10^8\) at \(p_\xi=0\) to \(2.47150365\times10^{11}\) at \(p_\xi=6\), i.e. by a factor of about \(7.48\times10^2\). The wall time grows from \(48.7720\,\mathrm{s}\) to \(25834.0100\,\mathrm{s}\), which is approximately \(7.18\) hours. The shape of Figure~\ref{fig:work} is monotone, as expected, because both \(N_{\mathrm{SG,slab}}\) and the number of preconditioner calls increase with \(p_\xi\). The local logarithmic growth of the work becomes milder for larger \(p_\xi\), reflecting the decreasing adjacent growth ratio of the total-degree basis dimension for fixed stochastic dimension \(M=4\).

The timing breakdown in Figure~\ref{fig:stochastic-refinement-timing-breakdown} shows that the solve time, wall time, and preconditioner time follow the same trend. The preconditioner dominates the linear solve. At \(p_\xi=6\), Table~\ref{tab:finite-size-work-solver-performance} gives \(\texttt{solve\_time}=22485.8000\,\mathrm{s}\) and \(\texttt{prec\_time}=21046.5600\,\mathrm{s}\); hence about \(93.6\%\) of the solve time is spent in preconditioner applications. By comparison, the operator-application time is \(1388.9440\,\mathrm{s}\). This identifies the block-Jacobi GMG preconditioner as the dominant computational cost component for the present implementation.

The finite benchmark has a sharp cutoff at \(p_\xi=q=6\). Hence the last point in the error plots should be interpreted as the exact finite-chaos capture point, not as part of an asymptotic stochastic decay regime. After this point, further stochastic enrichment would only increase the algebraic cost while the error remains controlled by the fixed deterministic discretization and the algebraic solver tolerance.

\section{Monte-Carlo approximation of the stochastic heat problem}
\label{sec:monte-carlo-method}

The intrusive stochastic Galerkin method developed above treats the random dependence by expanding the solution in the Hermite chaos basis and solving a single coupled deterministic system for all retained stochastic modes. In order to assess this approach and to put the stochastic Galerkin results into context, we also consider a classical Monte-Carlo strategy for the same stochastic heat problem introduced in \eqref{eq:heat_ic}.

The purpose of the Monte-Carlo method is different from that of the intrusive stochastic Galerkin method. Instead of constructing a global polynomial approximation in the stochastic variables, Monte-Carlo samples the random input repeatedly. For each realization of the random vector, the stochastic heat equation becomes a deterministic heat equation with a fixed diffusion coefficient. These deterministic problems are solved independently by the same space-time finite element method, and the stochastic quantities of interest defined in Subsection~\ref{subsec:qoi} are then approximated by empirical averages.

This makes Monte-Carlo a natural reference method. It is non-intrusive, straightforward to parallelize, and robust with respect to the stochastic dimension. Its main limitation is the slow sampling rate \(N_{\mathrm{MC}}^{-1/2}\), which is independent of the regularity of the solution with respect to the random variables; see \cite{Sullivan2015UQ,Caflisch1998MC,Xiu2010Book}. Hence, comparing Monte-Carlo with stochastic Galerkin allows us to compare two fundamentally different philosophies: sampling many deterministic realizations versus solving one coupled spectral problem in the stochastic space.

\subsection{Pathwise deterministic problem}
\label{subsec:mc-pathwise-problem}

Let
\(
\bm{\xi}^{(m)}
=
\left(\xi_1^{(m)},\ldots,\xi_M^{(m)}\right),
\qquad
m=1,\ldots,N_{\mathrm{MC}},
\)
be independent samples drawn from the product Gaussian density \(\rho\). For a fixed sample \(\bm{\xi}^{(m)}\), the random coefficient in \eqref{eq:heat_strong} becomes the deterministic coefficient \(
a^{(m)}(\bm{x}, t)
:=
a(\bm{x}, t, \bm{\xi}^{(m)}).
\)
Therefore, the pathwise deterministic problem is written as follows.
\begin{problem}[Pathwise deterministic Monte-Carlo problem]
\label{prob:mc-pathwise-deterministic}
For a fixed realization \(\bm{\xi}^{(m)}\), find
\[
u^{(m)}(\bm{x}, t)
:=
u(\bm{x}, t, \bm{\xi}^{(m)})
\]
such that
\begin{equation}
\partial_t u^{(m)}
-
\nabla_{\bm{x}}\cdot
\left(
a^{(m)}(\bm{x}, t)\nabla_{\bm{x}}u^{(m)}
\right)
=
f^{(m)}(\bm{x}, t)
\qquad
\text{in }(0,T]\times D,
\label{eq:mc-pathwise-strong}
\end{equation}
with the same initial and boundary conditions as in \eqref{eq:heat_ic}. Here
\[
f^{(m)}(\bm{x}, t)
:=
f(\bm{x}, t, \bm{\xi}^{(m)}).
\]

The corresponding pathwise weak formulation is: for almost every
\(t\in(0,T]\), find \(u^{(m)}(t)\in V:=H_0^1(D)\) such that
\begin{equation}
\left\langle
\partial_t u^{(m)}(t),v
\right\rangle_{V',V}
+
\int_D
a^{(m)}(\bm{x}, t)
\nabla_{\bm{x}}u^{(m)}(\bm{x}, t)
\cdot
\nabla_{\bm{x}}v(\bm{x})
\,d\bm{x}
=
\left(
f^{(m)}(t),v
\right)_{L^2(D)}
\label{eq:mc-pathwise-weak}
\end{equation}
for all \(v\in V\).
\end{problem}

\begin{definition}[Monte-Carlo sample problem]
\label{def:mc-sample-problem} A Monte-Carlo sample problem is the deterministic parabolic problem \eqref{eq:mc-pathwise-strong}--\eqref{eq:mc-pathwise-weak} obtained by fixing one realization \(\bm{\xi}^{(m)}\) of the random input. The Monte-Carlo method solves \(N_{\mathrm{MC}}\) such deterministic problems independently.
\end{definition}

\begin{remark}
The Monte-Carlo method does not require the Hermite triple-product matrices or the coupled stochastic Galerkin operator. Each sample is solved by the deterministic space-time solver already used inside the deterministic building blocks of the stochastic Galerkin method. This makes Monte-Carlo a useful reference implementation, but its convergence with respect to
\(N_{\mathrm{MC}}\) is only of order \(N_{\mathrm{MC}}^{-1/2}\).
\end{remark}

\paragraph{Space-time discretization of the Monte-Carlo samples}
\label{para:mc-spacetime-discretization}

For each realization \(\bm{\xi}^{(m)}\), the Monte-Carlo method uses the same deterministic space-time discretization as the stochastic Galerkin method. The spatial approximation is based on the conforming finite element space \(V_h\subset H_0^1(D)\) introduced in Subsection~\ref{subsec:sg-spatial-discretization}, while the time discretization uses the same scalar dG time space \(S_\tau\) and the identification
\(
T_\tau^{(r)}(V_h)\cong S_\tau\otimes V_h
\)
from Subsection~\ref{subsec:sg-time-discretization}. Thus, on each time slab
\(I_n=(t_{n-1},t_n]\), the sampled solution satisfies
\(
u_{h,\tau}^{(m)}\big|_{I_n}\in \mathbb P_r(I_n;V_h),
\)
and can be written locally as
\begin{equation}
u_{h,\tau}^{(m)}\big|_{I_n}(\bm{x},t)
=
\sum_{i=0}^{r}
\sum_{j=1}^{N_h}
U_{i,j}^{(m,n)}\,v_i^{(n)}(t)\,\phi_j(\bm{x}).
\label{eq:mc-local-expansion}
\end{equation}

The difference between the Monte-Carlo and stochastic Galerkin approaches is therefore not the deterministic space-time discretization, but the treatment of the stochastic variable. In the stochastic Galerkin method, the discrete space is enlarged to
\( X_{h,\tau,p_\xi}^{\mathrm{SG}}
=
S_\tau\otimes V_h\otimes \mathcal P_{p_\xi},
\)
and the retained Hermite modes are coupled through the stochastic Galerkin projection, as described in Subsections~\ref{subsec:sg-fully-discrete-space} and \ref{subsec:sg-slabwise-dg-formulation}. In the Monte-Carlo method, by contrast, no stochastic Galerkin space and no stochastic coupling matrices are introduced. For each sample \(\bm{\xi}^{(m)}\), one solves a deterministic space-time problem in
\(
X_{h,\tau}^{\mathrm{MC}}=S_\tau\otimes V_h
\)
with the sampled coefficient \(a(\bm{x},\bm{\xi}^{(m)})\) and the corresponding sampled right-hand side. Hence the randomness enters only through the sampled data and the resulting sample solution \(u_{h,\tau}^{(m)}\). The Monte-Carlo solves are independent over the sample index \(m\), whereas the stochastic Galerkin method solves one coupled system for all retained stochastic modes. Using the temporal matrices \(K_{\tau,n}\), \(M_{\tau,n}\), and \(C_{\tau,n}\) from the \(\mathrm{dG}(r)\) space-time formulation, the sample-wise algebraic system on \(I_n\) takes the tensor-product form
\begin{equation}
\left(
K_{\tau,n}\otimes M_h
+
M_{\tau,n}\otimes A_{h,n}^{(m)}
\right)
U_n^{(m)}
=
F_n^{(m)}
+
\left(
C_{\tau,n}\otimes M_h
\right)
U_{n-1}^{(m)} .
\label{eq:mc-slab-system}
\end{equation}
Here we use the same notation as in \eqref{eq:At}-\eqref{eq:local_rhs_tensor}, but without stochastic coupling. For each Monte-Carlo sample \(m\), the operator \(A_{h,n}^{(m)}\) is the deterministic spatial diffusion operator assembled with the sampled coefficient \(a(\bm{x},\bm{\xi}^{(m)})\), and the slab system is solved independently of all other samples.
\subsection{Monte-Carlo estimators for the quantities of interest}
\label{subsec:mc-estimators-qoi}

After solving the pathwise systems \eqref{eq:mc-slab-system} for \(m=1,\ldots,N_{\mathrm{MC}}\), we approximate the stochastic quantities of interest from Subsection~\ref{subsec:qoi}. At a fixed time \(t\), the Monte-Carlo estimator for the mean field \(\mathcal{M}[u]\) in
\eqref{eq:qoi-mean-definition} is
\begin{equation}
\widehat{\mathcal{M}}_{N_{\mathrm{MC}}}[u_{h,\tau}](\bm{x}, t)
:=
\frac{1}{N_{\mathrm{MC}}}
\sum_{m=1}^{N_{\mathrm{MC}}}
u_{h,\tau}^{(m)}(\bm{x}, t).
\label{eq:mc-mean-estimator}
\end{equation}
The corresponding unbiased estimator for the variance field
\(\mathcal{V}[u]\) in \eqref{eq:qoi-variance-definition} is
\begin{equation}
\widehat{\mathcal{V}}_{N_{\mathrm{MC}}}[u_{h,\tau}](\bm{x}, t)
:=
\frac{1}{N_{\mathrm{MC}}-1}
\sum_{m=1}^{N_{\mathrm{MC}}}
\left(
u_{h,\tau}^{(m)}(\bm{x}, t)
-
\widehat{\mathcal{M}}_{N_{\mathrm{MC}}}[u_{h,\tau}](\bm{x}, t)
\right)^2 .
\label{eq:mc-variance-estimator-unbiased}
\end{equation}
Equivalently, one often uses the second-moment form
\begin{equation}
\widehat{\mathcal{V}}_{N_{\mathrm{MC}}}^{\,\mathrm{biased}}[u_{h,\tau}]
=
\frac{1}{N_{\mathrm{MC}}}
\sum_{m=1}^{N_{\mathrm{MC}}}
\left(
u_{h,\tau}^{(m)}
\right)^2
-
\left(
\widehat{\mathcal{M}}_{N_{\mathrm{MC}}}[u_{h,\tau}]
\right)^2 .
\label{eq:mc-variance-estimator-biased}
\end{equation}
The difference between \eqref{eq:mc-variance-estimator-unbiased} and \eqref{eq:mc-variance-estimator-biased} is of order \(N_{\mathrm{MC}}^{-1}\). For statistical error analysis, we use the unbiased estimator \eqref{eq:mc-variance-estimator-unbiased}.

\begin{definition}[Monte-Carlo approximation of the QoIs]
\label{def:mc-qoi-approximation}
The Monte-Carlo approximations of the mean and variance fields are the empirical fields
\[
\widehat{\mathcal{M}}_{N_{\mathrm{MC}}}[u_{h,\tau}]
\quad\text{and}\quad
\widehat{\mathcal{V}}_{N_{\mathrm{MC}}}[u_{h,\tau}]
\]
defined by \eqref{eq:mc-mean-estimator} and
\eqref{eq:mc-variance-estimator-unbiased}. These are the Monte-Carlo
counterparts of the exact quantities \(\mathcal{M}[u]\) and
\(\mathcal{V}[u]\) defined in \eqref{eq:qoi-mean-definition} and
\eqref{eq:qoi-variance-definition}.
\end{definition}

The corresponding numerical Monte-Carlo errors are measured with the deterministic space-time norms from Subsection~\ref{subsec:error-norms}. Since the estimators in \eqref{eq:mc-mean-estimator} and \eqref{eq:mc-variance-estimator-unbiased} are formed from the fully discrete pathwise solutions, their direct comparison with the exact quantities \(\mathcal{M}[u]\) and \(\mathcal{V}[u]\) gives the total numerical Monte-Carlo error. In the \(L^2_t(L^2_x)\)-norm we define
\begin{align}
E_{N_{\mathrm{MC}},h,\tau}^{\mathrm{tot,mean},L^2}
:=
\left\|
\widehat{\mathcal{M}}_{N_{\mathrm{MC}}}[u_{h,\tau}]
-
\mathcal{M}[u]
\right\|_{L^2_t(L^2_x)}\  \ , \ \ \ 
E_{N_{\mathrm{MC}},h,\tau}^{\mathrm{tot,var},L^2}
:=
\left\|
\widehat{\mathcal{V}}_{N_{\mathrm{MC}}}[u_{h,\tau}]
-
\mathcal{V}[u]
\right\|_{L^2_t(L^2_x)}.
\label{eq:mc-error-mean-L2,mc-error-var-L2}
\end{align}
Analogous \(H^1\)-seminorm versions are defined using
\eqref{eq:deterministic-L2t-H1semi-norm} whenever the variance field is
sufficiently smooth.

\subsection{Sampling error and root-\(N_{\mathrm{MC}}\) convergence}
\label{subsec:mc-root-n}

The central theoretical property of standard Monte-Carlo sampling is that sampling errors decay like \(N_{\mathrm{MC}}^{-1/2}\). This rate is independent of the stochastic dimension, but it is also slow compared with the spectral or stretched-exponential decay that one may obtain from stochastic Galerkin when the solution depends smoothly on the random variables.

The following result states the classical root-\(N_{\mathrm{MC}}\) behavior for the Monte-Carlo mean estimator; see \cite{Sullivan2015UQ,Caflisch1998MC,Xiu2010Book}. It applies to Hilbert-space-valued random variables and therefore covers field-valued quantities such as \(u(\cdot,t,\bm{\xi})\), or after integration in time, the corresponding space-time fields. The following Lemma is a standard Hilbert-space-valued Monte-Carlo mean-square estimate (see, e.g., \cite[Lemma~2.1]{barthLang2012}).

\begin{lemma}
\label{lem:mc-mean-rmse}
Let \(U\) be a Hilbert-space-valued random variable with \(U\in L^2(\Gamma,\rho;X)\), where \(X\) is a Hilbert space such as \(L^2(D)\) or \(L^2(0,T;L^2(D))\). Let
\[
\mathcal{M}[U]=\mathbb{E}_{\rho}[U].
\]
For independent samples \(U^{(m)}=U(\bm{\xi}^{(m)})\), define
\[
\widehat{\mathcal{M}}_{N_{\mathrm{MC}}}[U]
=
\frac{1}{N_{\mathrm{MC}}}
\sum_{m=1}^{N_{\mathrm{MC}}}
U^{(m)}.
\]
Then
\[
\mathbb{E}_{\rho}
\left[
\widehat{\mathcal{M}}_{N_{\mathrm{MC}}}[U]
\right]
=
\mathcal{M}[U],
\]
and
\begin{equation}
\mathbb{E}_{\rho}
\left[
\left\|
\widehat{\mathcal{M}}_{N_{\mathrm{MC}}}[U]
-
\mathcal{M}[U]
\right\|_{X}^2
\right]
=
\frac{1}{N_{\mathrm{MC}}}
\mathbb{E}_{\rho}
\left[
\left\|
U-\mathcal{M}[U]
\right\|_{X}^2
\right].
\label{eq:mc-mean-rmse-rootN}
\end{equation}
Consequently,
\[
\operatorname{RMSE}
\left(
\widehat{\mathcal{M}}_{N_{\mathrm{MC}}}[U]
\right)
=
O(N_{\mathrm{MC}}^{-1/2}).
\]
\end{lemma}

\begin{proof}
Unbiasedness follows immediately from linearity of expectation. For the
mean-square error, write
\[
U^{(m)}-\mathcal{M}[U]
\]
as independent, centered Hilbert-space-valued random variables. Expanding the
squared norm of their average, all cross terms vanish by independence and
zero mean. Therefore,
\[
\mathbb{E}_{\rho}
\left[
\left\|
\frac{1}{N_{\mathrm{MC}}}
\sum_{m=1}^{N_{\mathrm{MC}}}
\left(
U^{(m)}-\mathcal{M}[U]
\right)
\right\|_{X}^2
\right]
=
\frac{1}{N_{\mathrm{MC}}^2}
\sum_{m=1}^{N_{\mathrm{MC}}}
\mathbb{E}_{\rho}
\left[
\left\|
U^{(m)}-\mathcal{M}[U]
\right\|_{X}^2
\right],
\]
which gives \eqref{eq:mc-mean-rmse-rootN}.
\end{proof}

The unbiasedness and root-mean-square behavior of the sample variance estimator follow from the classical fourth-moment formula for the sample variance, (see, e.g., \cite{kenneyKeeping1951,bierigChernov2016}).
\begin{lemma}
\label{lem:mc-variance-rmse}
Assume that the random field has a finite fourth moment. Then the sample variance estimator \eqref{eq:mc-variance-estimator-unbiased} is pointwise unbiased for \(\mathcal{V}[u_{h,\tau}]\), and its root-mean-square sampling error satisfies
\[
\operatorname{RMSE}
\left(
\widehat{\mathcal{V}}_{N_{\mathrm{MC}}}[u_{h,\tau}]
\right)
=
O(N_{\mathrm{MC}}^{-1/2}).
\]
The constant depends on fourth central moments of the sampled field.
\end{lemma}

\subsection{Discretization, sampling, and total Monte-Carlo errors}
\label{subsec:mc-error-decomposition}

The total errors defined in \eqref{eq:mc-error-mean-L2,mc-error-var-L2} and \eqref{eq:mc-error-mean-L2,mc-error-var-L2} contain two distinct effects. The first is the deterministic space-time discretization error of the pathwise sample problems. The second is the statistical sampling error caused by using only finitely many random samples. For a meaningful interpretation of the numerical results, these contributions have to be separated.

This separation is especially useful for prescribed solutions. In that case the exact sampled solution \(u(\bm{x},t,\bm{\xi}^{(m)})\) is known analytically, so one can form Monte-Carlo estimators both from the exact sampled fields and from the fully discrete fields. We define the exact sampled mean by
\begin{equation}
\widehat{\mathcal{M}}_{N_{\mathrm{MC}}}[u](\bm{x},t)
:=
\frac{1}{N_{\mathrm{MC}}}
\sum_{m=1}^{N_{\mathrm{MC}}}
u(\bm{x},t,\bm{\xi}^{(m)}),
\label{eq:mc-exact-sampled-mean-estimator}
\end{equation}
and the exact sampled variance by
\begin{equation}
\widehat{\mathcal{V}}_{N_{\mathrm{MC}}}[u](\bm{x},t)
:=
\frac{1}{N_{\mathrm{MC}}-1}
\sum_{m=1}^{N_{\mathrm{MC}}}
\left(
u(\bm{x},t,\bm{\xi}^{(m)})
-
\widehat{\mathcal{M}}_{N_{\mathrm{MC}}}[u](\bm{x},t)
\right)^2 .
\label{eq:mc-exact-sampled-var-estimator}
\end{equation}
These are the exact-solution counterparts of the discrete estimators
\(\widehat{\mathcal{M}}_{N_{\mathrm{MC}}}[u_{h,\tau}]\) and
\(\widehat{\mathcal{V}}_{N_{\mathrm{MC}}}[u_{h,\tau}]\) defined in
\eqref{eq:mc-mean-estimator} and
\eqref{eq:mc-variance-estimator-unbiased}.

For the mean field, we define the discretization error on the same Monte-Carlo sample set by
\begin{equation}
E_{N_{\mathrm{MC}},h,\tau}^{\mathrm{disc,mean},L^2}
:=
\left\|
\widehat{\mathcal{M}}_{N_{\mathrm{MC}}}[u_{h,\tau}]
-
\widehat{\mathcal{M}}_{N_{\mathrm{MC}}}[u]
\right\|_{L^2_t(L^2_x)}.
\label{eq:mc-disc-mean-error}
\end{equation}
The pure Monte-Carlo sampling error of the exact sampled mean is
\begin{equation}
E_{N_{\mathrm{MC}}}^{\mathrm{MC,mean},L^2}
:=
\left\|
\widehat{\mathcal{M}}_{N_{\mathrm{MC}}}[u]
-
\mathcal{M}[u]
\right\|_{L^2_t(L^2_x)}.
\label{eq:mc-pure-sampling-mean-error}
\end{equation}
Together with the total mean error
\(E_{N_{\mathrm{MC}},h,\tau}^{\mathrm{tot,mean},L^2}\) already defined in
\eqref{eq:mc-error-mean-L2,mc-error-var-L2}, these are the three mean-error quantities used in
the Monte-Carlo study.

Analogously, for the variance field, the discretization error on the sampled
variance field is
\begin{equation}
E_{N_{\mathrm{MC}},h,\tau}^{\mathrm{disc,var},L^2}
:=
\left\|
\widehat{\mathcal{V}}_{N_{\mathrm{MC}}}[u_{h,\tau}]
-
\widehat{\mathcal{V}}_{N_{\mathrm{MC}}}[u]
\right\|_{L^2_t(L^2_x)}.
\label{eq:mc-disc-var-error}
\end{equation}
The pure Monte-Carlo sampling error of the exact sampled variance is
\begin{equation}
E_{N_{\mathrm{MC}}}^{\mathrm{MC,var},L^2}
:=
\left\|
\widehat{\mathcal{V}}_{N_{\mathrm{MC}}}[u]
-
\mathcal{V}[u]
\right\|_{L^2_t(L^2_x)}.
\label{eq:mc-pure-sampling-var-error}
\end{equation}
Together with the total variance error
\(E_{N_{\mathrm{MC}},h,\tau}^{\mathrm{tot,var},L^2}\) already defined in
\eqref{eq:mc-error-mean-L2,mc-error-var-L2}, these are the three variance-error quantities used
in the Monte-Carlo study.

The total mean error admits the exact sample-wise splitting
\begin{equation}
\begin{aligned}
\widehat{\mathcal{M}}_{N_{\mathrm{MC}}}[u_{h,\tau}]
-
\mathcal{M}[u]
=
\left(
\widehat{\mathcal{M}}_{N_{\mathrm{MC}}}[u_{h,\tau}]
-
\widehat{\mathcal{M}}_{N_{\mathrm{MC}}}[u]
\right)
\quad+
\left(
\widehat{\mathcal{M}}_{N_{\mathrm{MC}}}[u]
-
\mathcal{M}[u]
\right).
\end{aligned}
\label{eq:mc-total-mean-sample-splitting}
\end{equation}
Therefore, by the triangle inequality,
\begin{equation}
E_{N_{\mathrm{MC}},h,\tau}^{\mathrm{tot,mean},L^2}
\le
E_{N_{\mathrm{MC}},h,\tau}^{\mathrm{disc,mean},L^2}
+
E_{N_{\mathrm{MC}}}^{\mathrm{MC,mean},L^2}.
\label{eq:mc-total-mean-triangle}
\end{equation}
Similarly, adding and subtracting the exact sampled variance field gives
\begin{equation}
\begin{aligned}
\widehat{\mathcal{V}}_{N_{\mathrm{MC}}}[u_{h,\tau}]
-
\mathcal{V}[u] =
\left(
\widehat{\mathcal{V}}_{N_{\mathrm{MC}}}[u_{h,\tau}]
-
\widehat{\mathcal{V}}_{N_{\mathrm{MC}}}[u]
\right) \quad+
\left(
\widehat{\mathcal{V}}_{N_{\mathrm{MC}}}[u]
-
\mathcal{V}[u]
\right),
\end{aligned}
\label{eq:mc-total-var-sample-splitting}
\end{equation}
and hence
\begin{equation}
E_{N_{\mathrm{MC}},h,\tau}^{\mathrm{tot,var},L^2}
\le
E_{N_{\mathrm{MC}},h,\tau}^{\mathrm{disc,var},L^2}
+
E_{N_{\mathrm{MC}}}^{\mathrm{MC,var},L^2}.
\label{eq:mc-total-var-triangle}
\end{equation}
Although the variance estimator is nonlinear in the sampled fields, the decomposition above is still valid because it is obtained by adding and subtracting the same exact sampled variance field.

The pure sampling errors
\(E_{N_{\mathrm{MC}}}^{\mathrm{MC,mean},L^2}\) and \(E_{N_{\mathrm{MC}}}^{\mathrm{MC,var},L^2}\) are the quantities expected to exhibit the classical Monte-Carlo rate $O(N_{\mathrm{MC}}^{-1/2}) $ in root-mean-square, according to Lemma~\ref{lem:mc-mean-rmse} and Lemma~\ref{lem:mc-variance-rmse}. In contrast, the discretization errors \(E_{N_{\mathrm{MC}},h,\tau}^{\mathrm{disc,mean},L^2}\) and \(E_{N_{\mathrm{MC}},h,\tau}^{\mathrm{disc,var},L^2}\) are controlled by the deterministic space-time discretization parameters \(h\), \(\tau\), and the polynomial degrees in space and time. The total errors \(E_{N_{\mathrm{MC}},h,\tau}^{\mathrm{tot,mean},L^2}\) and \(E_{N_{\mathrm{MC}},h,\tau}^{\mathrm{tot,var},L^2}\) contain both effects and therefore reveal which contribution dominates the actual numerical Monte-Carlo approximation.

The preceding definitions are sample-wise and are therefore directly computable in the prescribed analytical solution's benchmark. For theoretical error estimates, it is also useful to formulate the decomposition at the expectation level. Let
\[
m(\bm{x}, t)
:=
\mathcal{M}[u](\bm{x}, t),
\ \ \text{and } \
m_{h,\tau}(\bm{x}, t)
:=
\mathcal{M}[u_{h,\tau}](\bm{x}, t)
\]
be the exact mean of the continuous solution and the exact mean of the fully discrete pathwise solution respectively. Then
\begin{equation}
\begin{aligned}
\widehat{\mathcal{M}}_{N_{\mathrm{MC}}}[u_{h,\tau}]
-
m
=
\underbrace{
\left(
m_{h,\tau}-m
\right)
}_{\text{deterministic discretization bias}}
\quad+
\underbrace{
\left(
\widehat{\mathcal{M}}_{N_{\mathrm{MC}}}[u_{h,\tau}]
-
m_{h,\tau}
\right)
}_{\text{Monte-Carlo sampling error of the discrete field}} .
\end{aligned}
\label{eq:mc-mean-error-splitting}
\end{equation}
Consequently,
\begin{equation}
\left(
\mathbb{E}_{\rho}
\left[
\left\|
\widehat{\mathcal{M}}_{N_{\mathrm{MC}}}[u_{h,\tau}]
-
m
\right\|_{L^2_t(L^2_x)}^2
\right]
\right)^{1/2}
\le
\|m_{h,\tau}-m\|_{L^2_t(L^2_x)}
+
\frac{C_{h,\tau}}{\sqrt{N_{\mathrm{MC}}}}.
\label{eq:mc-bias-sampling-bound}
\end{equation}
The first term is controlled by mesh refinement in \(h\) and \(\tau\), while the second term is controlled by the sample size \(N_{\mathrm{MC}}\). An analogous expectation-level decomposition holds for the variance error, with a constant depending on fourth central moments of the sampled field.

\begin{definition}[Sampling-dominated regime]
\label{def:sampling-dominated-regime}
We say that a Monte-Carlo experiment is in the sampling-dominated regime if
the deterministic bias $\|m_{h,\tau}-m\|_{L^2_t(L^2_x)}$ is much smaller than the sampling term \(C_{h,\tau}N_{\mathrm{MC}}^{-1/2}\).
\end{definition}

In this regime, increasing the space-time resolution does not visibly reduce the total error, and log--log plots of the error against \(N_{\mathrm{MC}}\) should reveal the theoretical slope \(-1/2\). Conversely, if the sampling error is much smaller than the deterministic space-time bias, then increasing \(N_{\mathrm{MC}}\) does not substantially reduce the total error. In that case the experiment is discretizatio dominated, and the total error can only be reduced by refining the deterministic space-time approximation.

\begin{remark}[Practical verification strategy]
For prescribed solutions, the exact mean and variance are often available analytically. Then one can verify the deterministic discretization by refining \(h\) and \(\tau\) while monitoring the discretization errors \eqref{eq:mc-disc-mean-error} and \eqref{eq:mc-disc-var-error}. The sampling rate is verified by fixing a sufficiently fine deterministic discretization and increasing \(N_{\mathrm{MC}}\), while monitoring the pure sampling errors \eqref{eq:mc-pure-sampling-mean-error} and \eqref{eq:mc-pure-sampling-var-error}. Finally, the total errors \eqref{eq:mc-error-mean-L2,mc-error-var-L2} and \eqref{eq:mc-error-mean-L2,mc-error-var-L2} show the accuracy of the actual numerical Monte-Carlo approximation. A comprehensive Monte-Carlo study should therefore report all three quantities: the discretization error, the pure sampling error, and the total error.
\end{remark}

\subsection{Monte-Carlo algorithm}
\label{subsec:mc-algorithm}

The Monte-Carlo procedure used in the numerical experiments follows directly from the pathwise formulation \eqref{eq:mc-pathwise-strong}--\eqref{eq:mc-pathwise-weak} and the sample-wise space-time discretization mentioned in Section-\ref{para:mc-spacetime-discretization}. For each realization of the random vector, one deterministic space-time finite element problem is solved. The resulting degrees of freedom are then used to construct the empirical quantities of interest and the corresponding error estimators.

The main steps are summarized in Algorithm~\ref{alg:mc-fem}. The algorithm is written at the level of the numerical workflow rather than at the level of low-level implementation details.

\begin{algorithm}[H]
\caption{Monte-Carlo space-time finite element procedure}
\label{alg:mc-fem}
\begin{algorithmic}[1]
\Require Number of samples \(N_{\mathrm{MC}}\), mesh parameter \(h\), time step
\(\tau\), polynomial degrees, stochastic dimension \(M\), and random density
\(\rho\).
\Ensure Empirical mean and variance fields, together with discretization,
pure sampling, and total error estimators.
\State Construct the spatial mesh \(\mathcal{T}_h\), the finite element space
\(V_h\), and the dG\((r)\) time discretization used in
\eqref{eq:mc-local-expansion}.
\State Initialize the accumulators for the empirical mean and variance fields
defined in \eqref{eq:mc-mean-estimator} and
\eqref{eq:mc-variance-estimator-unbiased}.
\For{\(m=1,\ldots,N_{\mathrm{MC}}\)}
    \State Draw an independent sample $\bm{\xi}^{(m)}\sim \rho .$
    \State Form the sampled coefficient \(a^{(m)}\) and sampled right-hand side
    \(f^{(m)}\) as in \eqref{eq:mc-pathwise-strong}.
    \State For each time slab \(I_n\), assemble and solve the deterministic
    slab system \eqref{eq:mc-slab-system}.
    \State Store or accumulate the resulting space-time degrees of freedom
    \(U_n^{(m)}\), which define the pathwise finite element solution
    \(u_{h,\tau}^{(m)}\) through \eqref{eq:mc-local-expansion}.
    \State If the prescribed exact solution is available, evaluate the
    corresponding exact sampled field \(u(\bm{x},t,\bm{\xi}^{(m)})\) on the
    same space-time quadrature structure.
\EndFor
\State Compute the empirical mean and variance fields $ \widehat{\mathcal{M}}_{N_{\mathrm{MC}}}[u_{h,\tau}], \ \widehat{\mathcal{V}}_{N_{\mathrm{MC}}}[u_{h,\tau}],$
using \eqref{eq:mc-mean-estimator} and \eqref{eq:mc-variance-estimator-unbiased} and $\widehat{\mathcal{M}}_{N_{\mathrm{MC}}}[u],\ \widehat{\mathcal{V}}_{N_{\mathrm{MC}}}[u],$ defined in \eqref{eq:mc-exact-sampled-mean-estimator} and \eqref{eq:mc-exact-sampled-var-estimator}.
\State Evaluate the discretization errors \eqref{eq:mc-disc-mean-error} and \eqref{eq:mc-disc-var-error}, the pure sampling errors \eqref{eq:mc-pure-sampling-mean-error} and
\eqref{eq:mc-pure-sampling-var-error}, and the total errors \eqref{eq:mc-error-mean-L2,mc-error-var-L2} and \eqref{eq:mc-error-mean-L2,mc-error-var-L2}.
\State \textbf{Return:} Report the results of the QoI. 
\end{algorithmic}
\end{algorithm}

The role of the degrees of freedom \(U_n^{(m)}\) is central. For each sample, they represent the pathwise finite element solution \(u_{h,\tau}^{(m)}\) on the space-time slab through \eqref{eq:mc-local-expansion}. The empirical mean and variance are therefore not computed from scalar sample outputs, but from the sampled finite element fields themselves. This makes the Monte-Carlo QoIs field-valued approximations of \(\mathcal{M}[u]\) and \(\mathcal{V}[u]\), as defined in \eqref{eq:qoi-mean-definition} and \eqref{eq:qoi-variance-definition}.

For the prescribed benchmark, the exact sampled fields are also available. This makes it possible to distinguish the three error mechanisms introduced in Subsection~\ref{subsec:mc-error-decomposition}. The discretization errors measure the difference between numerical and exact sampled estimators on the same sample set. The pure sampling errors measure the difference between the exact sampled estimators and the analytical QoIs. The total errors compare the numerical Monte-Carlo estimators directly with the analytical QoIs. This separation is essential for identifying whether a computation is dominated by space-time discretization error or by Monte-Carlo sampling error.\\

The Monte-Carlo method provides a complementary non-intrusive baseline for the intrusive stochastic Galerkin method. Both methods are applied to the same stochastic heat problem \eqref{eq:heat_ic}, but they represent the stochastic dependence in fundamentally different ways.

First, Monte-Carlo solves many independent pathwise deterministic heat problems. For each sample of the random input, the stochastic problem reduces to one deterministic space-time problem of the form \eqref{eq:mc-pathwise-strong}--\eqref{eq:mc-pathwise-weak}. The stochastic quantities are then recovered statistically from the ensemble of pathwise solutions. In particular, Monte-Carlo directly approximates the mean \(\mathcal{M}[u]\) defined in \eqref{eq:qoi-mean-definition} and the variance \(\mathcal{V}[u]\) defined in \eqref{eq:qoi-variance-definition} by the estimators \eqref{eq:mc-mean-estimator} and \eqref{eq:mc-variance-estimator-unbiased}.

This is different from the stochastic Galerkin method. In the stochastic Galerkin approach, the solution is represented by Hermite chaos modes, and therefore the mean and variance are obtained directly from the modal representation through \eqref{eq:qoi-mean-chaos} and \eqref{eq:qoi-variance-chaos}. Moreover, stochastic Galerkin also gives access to the full stochastic error measures \eqref{eq:error-full-L2}--\eqref{eq:error-full-H1}, because the full \(L^2_\rho\)-structure is represented explicitly in the Galerkin expansion. The Monte-Carlo comparison, however, is primarily performed at the level of the statistical quantities of interest, namely the mean and the variance, not at the level of the full stochastic Galerkin \(L^2_\rho\)-error norms.

Second, Monte-Carlo has the universal sampling rate $\operatorname{RMSE} = O(N_{\mathrm{MC}}^{-1/2}),$ as discussed in Lemma~\ref{lem:mc-mean-rmse} andLemma~\ref{lem:mc-variance-rmse}. This rate is independent of the stochastic dimension, but it is also slow. By contrast, the stochastic Galerkin method can show much faster decay with respect to the stochastic polynomial degree \(p_\xi\) when the prescribed solution has smooth (or analytic) dependence on the random variables.

Third, the two methods have different computational cost structures. Monte-Carlo requires many independent deterministic solves, whereas stochastic Galerkin requires one larger coupled spectral solve in the stochastic modes. Thus the comparison is not only a comparison of accuracy, but also a comparison of computational strategies:
\[
\text{many independent deterministic pathwise solves}
\qquad
\text{versus}
\qquad
\text{one coupled stochastic Galerkin solve}.
\]

In the following section, we present the numerical Monte-Carlo results and compare them with the stochastic Galerkin results at the level of the mean and variance quantities of interest. This comparison shows whether the two methods produce consistent statistical predictions for the same stochastic heat problem, while also highlighting their different convergence mechanisms: Monte-Carlo converges statistically with the classical \(N_{\mathrm{MC}}^{-1/2}\) rate, whereas stochastic Galerkin exploits the regularity of the stochastic dependence through its Hermite polynomial approximation.\\

\subsection{Monte-Carlo Numerical Results}
\label{subsec:mc-sampling-rate-diagnostics}

We now study the two error mechanisms of the Monte-Carlo approximation: the statistical sampling error and the deterministic space-time discretization error. The cumulative sampling diagnostics are evaluated on the finest deterministic refinement level,
\(
h=\tau=0.03125,
\)
for the sample windows
\[
N_{\mathrm{MC}}=100,\ 250,\ 500,\ 750,\ 1000, \ 2000, \ \ldots \ , 5000.
\]
This tests whether the observed Monte-Carlo errors are compatible with the classical root-sampling law. In addition, the deterministic refinement study uses five refinement levels,
\(
h=\tau\in\{0.5,\ 0.25,\ 0.125,\ 0.0625,\ 0.03125\},
\)
with \(N_{\mathrm{MC}}=1000\) samples. This second study separates the effect of deterministic space-time refinement from the finite-sample Monte-Carlo error.

To test the theoretical Monte-Carlo rate, we use the model
\begin{equation}
E_N \approx C N^{-1/2}.
\label{eq:mc-rootN-fit-model}
\end{equation}
For a fixed reference sample size \(N_0=100\), the observed rate over the window \(100\to N\) is computed as
\begin{equation}
r_{100\to N}
:=
\frac{\log(E_{100}/E_N)}{\log(N/100)}.
\label{eq:mc-observed-rate-window}
\end{equation}
The corresponding empirical constant in \eqref{eq:mc-rootN-fit-model} is \(C_N:=E_N\sqrt{N}\). For ideal \(N^{-1/2}\)-decay, the observed rates \(r_{100\to N}\) would be close to \(1/2\), and the constants \(C_N\) would be approximately independent of \(N\). Since the reported values come from one cumulative Monte-Carlo realization, however, monotone decay should not be expected. Oscillations are especially natural for the variance estimator, because it depends on second moments and its sampling constant is controlled by fourth central moments.

Table~\ref{tab:mc-total-exact-mean-var-rootN} reports the root-sampling
diagnostics for the total and pure exact-sampling mean and variance errors on
the finest refinement level \(h=\tau=0.03125\). The total and exact-sampling
columns agree to the displayed precision. This confirms that, on this
refinement level, the deterministic space-time discretization error is
negligible compared with the sampling error.

\begin{table}
\caption{Monte-Carlo root-sampling diagnostics for the total and pure
exact-sampling mean and variance errors for the extended
\(N_{\mathrm{MC}}=5000\) run at \(h=\tau=0.125\). The rates are computed with
respect to the reference value \(N_{\mathrm{MC}}=100\) using
\eqref{eq:mc-observed-rate-window}. The measured error quantities correspond
to \eqref{eq:mc-error-mean-L2,mc-error-var-L2}, \eqref{eq:mc-pure-sampling-mean-error} and \eqref{eq:mc-pure-sampling-var-error}}
\label{tab:mc-total-exact-mean-var-rootN}

\scriptsize
\setlength{\tabcolsep}{2.2pt}
\renewcommand{\arraystretch}{1.08}

\resizebox{\linewidth}{!}{%
\begin{tabular}{@{}rrrrrrrrr@{}}
\hline
& \multicolumn{4}{c}{Mean error}
& \multicolumn{4}{c}{Variance error} \\
\cline{2-5}\cline{6-9}
\(N_{\mathrm{MC}}\)
& \(E_N^{\mathrm{tot}}\)
& \(r^{\mathrm{tot}}_{100\to N}\)
& \(E_N^{\mathrm{ex}}\)
& \(r^{\mathrm{ex}}_{100\to N}\)
& \(E_N^{\mathrm{tot}}\)
& \(r^{\mathrm{tot}}_{100\to N}\)
& \(E_N^{\mathrm{ex}}\)
& \(r^{\mathrm{ex}}_{100\to N}\) \\
\hline
100
& \(3.2621\times10^{-3}\)
& --
& \(3.2621\times10^{-3}\)
& --
& \(9.9018\times10^{-3}\)
& --
& \(9.9018\times10^{-3}\)
& -- \\
250
& \(6.5449\times10^{-3}\)
& \(-0.7599\)
& \(6.5449\times10^{-3}\)
& \(-0.7599\)
& \(2.2713\times10^{-3}\)
& \(1.6069\)
& \(2.2713\times10^{-3}\)
& \(1.6069\) \\
500
& \(3.2110\times10^{-3}\)
& \(0.0098\)
& \(3.2110\times10^{-3}\)
& \(0.0098\)
& \(1.5289\times10^{-3}\)
& \(1.1608\)
& \(1.5289\times10^{-3}\)
& \(1.1608\) \\
750
& \(1.2228\times10^{-3}\)
& \(0.4870\)
& \(1.2228\times10^{-3}\)
& \(0.4870\)
& \(1.9128\times10^{-3}\)
& \(0.8160\)
& \(1.9128\times10^{-3}\)
& \(0.8160\) \\
1000
& \(1.1219\times10^{-3}\)
& \(0.4636\)
& \(1.1219\times10^{-3}\)
& \(0.4636\)
& \(8.9699\times10^{-4}\)
& \(1.0429\)
& \(8.9699\times10^{-4}\)
& \(1.0429\) \\
2000
& \(2.3828\times10^{-3}\)
& \(0.1048\)
& \(2.3828\times10^{-3}\)
& \(0.1048\)
& \(1.8082\times10^{-3}\)
& \(0.5676\)
& \(1.8082\times10^{-3}\)
& \(0.5676\) \\
3000
& \(7.3660\times10^{-4}\)
& \(0.4375\)
& \(7.3660\times10^{-4}\)
& \(0.4375\)
& \(9.1144\times10^{-4}\)
& \(0.7014\)
& \(9.1144\times10^{-4}\)
& \(0.7014\) \\
4000
& \(5.5469\times10^{-4}\)
& \(0.4803\)
& \(5.5469\times10^{-4}\)
& \(0.4803\)
& \(6.4243\times10^{-4}\)
& \(0.7415\)
& \(6.4243\times10^{-4}\)
& \(0.7415\) \\
5000
& \(4.4384\times10^{-4}\)
& \(0.5099\)
& \(4.4384\times10^{-4}\)
& \(0.5099\)
& \(9.4659\times10^{-5}\)
& \(1.1887\)
& \(9.4660\times10^{-5}\)
& \(1.1887\) \\
\hline
\end{tabular}
}
\end{table}

We first read Table~\ref{tab:mc-total-exact-mean-var-rootN} as a sampling-rate diagnostic on the finest deterministic grid, \(h=\tau=0.03125\). The table compares the total Monte-Carlo error with the pure exact-sampling error at the same sample milestones. The two columns agree to the displayed precision, so the deterministic space-time error is not visible at this refinement level; the measured error is governed by sampling.

The following observations are relevant for the mean estimator:
\begin{itemize}
  \item The mean error is not monotone over the cumulative sample sequence. It increases from \(3.2621\times10^{-3}\) at \(N_{\mathrm{MC}}=100\) to \(6.5449\times10^{-3}\) at \(N_{\mathrm{MC}}=250\). This gives a negative window rate, but this is a feature of one finite cumulative Monte-Carlo realization, not a contradiction of the root-sampling model.
  \item For the larger windows, the mean error is closer to the expected sampling scale. The observed rates are \(r_{100\to750}^{\mathrm{mean}}=0.4870\) and \(r_{100\to1000}^{\mathrm{mean}}=0.4636\), which are close to the theoretical value \(1/2\).
  \item The final mean error at \(N_{\mathrm{MC}}=1000\) is \(1.1219\times10^{-3}\).
\end{itemize}

For the variance estimator, Table~\ref{tab:mc-total-exact-mean-var-rootN} shows stronger fluctuations:
\begin{itemize}
  \item The variance error decreases from \(9.9018\times10^{-3}\) at \(N_{\mathrm{MC}}=100\) to \(2.2713\times10^{-3}\) at \(N_{\mathrm{MC}}=250\), then to \(1.5289\times10^{-3}\) at \(N_{\mathrm{MC}}=500\), increases again at \(N_{\mathrm{MC}}=750\), and finally reaches \(8.9699\times10^{-4}\) at \(N_{\mathrm{MC}}=1000\).
  \item This stronger oscillation is expected for a variance estimator, since the variance is a second-moment quantity and its statistical error depends on higher moments; see Lemma~\ref{lem:mc-variance-rmse}.
  \item The final variance error at \(N_{\mathrm{MC}}=1000\) is \(8.9699\times10^{-4}\).
\end{itemize}

Overall, Table~\ref{tab:mc-total-exact-mean-var-rootN} confirms that the finest-grid Monte-Carlo errors are sampling errors. The observed rates should therefore be interpreted as finite-window diagnostics of one cumulative sample sequence, not as monotone convergence rates.\\

Table~\ref{tab:mc-disc-error-rates} reports the observed rates for the discretization-only errors under the refinement sequence \(h_\ell\mapsto h_{\ell+1}=h_\ell/2\). These rates describe the deterministic pathwise solver and are independent of the sampling fluctuation in the total Monte-Carlo estimator.\\

\begin{table}
\caption{Observed space-time refinement rates for the discretization-only
Monte-Carlo errors at \(N_{\mathrm{MC}}=1000\). The refinements satisfy
\(h=\tau\), and the rate is computed with respect to
\(h_\ell\mapsto h_{\ell+1}=h_\ell/2\), namely
\(r_{\ell\to \ell+1}=\log_2(E_\ell/E_{\ell+1})\).}
\label{tab:mc-disc-error-rates}

\scriptsize
\setlength{\tabcolsep}{2.8pt}
\renewcommand{\arraystretch}{1.08}

\resizebox{\linewidth}{!}{%
\begin{tabular}{@{}lrrrrrrrrr@{}}
\hline
error quantity
& \(E_1\)
& \(r_{1\to2}\)
& \(E_2\)
& \(r_{2\to3}\)
& \(E_3\)
& \(r_{3\to4}\)
& \(E_4\)
& \(r_{4\to5}\)
& \(E_5\) \\
\hline
\(E_{\mathrm{disc}}^{\mathrm{mean},L^2}\)
& \(3.267518\times10^{3}\)
& \(16.378\)
& \(3.837484\times10^{-2}\)
& \(14.805\)
& \(1.340143\times10^{-6}\)
& \(4.000\)
& \(8.377028\times10^{-8}\)
& \(4.000\)
& \(5.235821\times10^{-9}\) \\
\(E_{\mathrm{disc}}^{\mathrm{mean},H^1}\)
& \(1.116132\times10^{4}\)
& \(15.545\)
& \(2.334227\times10^{-1}\)
& \(13.669\)
& \(1.791646\times10^{-5}\)
& \(3.014\)
& \(2.217301\times10^{-6}\)
& \(3.004\)
& \(2.764635\times10^{-7}\) \\
\(E_{\mathrm{disc}}^{\mathrm{var},L^2}\)
& \(8.092823\times10^{10}\)
& \(31.854\)
& \(2.084483\times10^{1}\)
& \(27.322\)
& \(1.242715\times10^{-7}\)
& \(3.994\)
& \(7.797166\times10^{-9}\)
& \(3.999\)
& \(4.878019\times10^{-10}\) \\
\(E_{\mathrm{disc}}^{\mathrm{var},H^1}\)
& \(3.061790\times10^{11}\)
& \(30.789\)
& \(1.649952\times10^{2}\)
& \(25.524\)
& \(3.418897\times10^{-6}\)
& \(3.003\)
& \(4.263293\times10^{-7}\)
& \(3.001\)
& \(5.325844\times10^{-8}\) \\
\hline
\end{tabular}
}
\end{table}

From Table~\ref{tab:mc-disc-error-rates} we observe the deterministic convergence pattern of the pathwise solver:
\begin{itemize}
  \item For the \(L^2_t(L^2_x)\)-type quantities, the last two displayed rates are essentially fourth order. The mean discretization-only rates are \(4.000\) and \(4.000\), and the variance discretization-only rates are \(3.994\) and \(3.999\).
  \item For the \(L^2_t(H^1_x)\)-seminorm quantities, the last two displayed rates are essentially third order. The mean rates are \(3.014\) and \(3.004\), and the variance rates are \(3.003\) and \(3.001\).
\end{itemize}

From Table \ref{tab:mc-disc-error-rates}, we conclude that reducing \(h\) and \(\tau\) is effective for reducing the pathwise discretization error, but it does not reduce the total Monte-Carlo error once the finite-sample contribution dominates. In that regime, improving the total Monte-Carlo accuracy requires increasing \(N_{\mathrm{MC}}\).

\begin{figure}[tbp]
\centering
\includegraphics[width=0.7\linewidth]{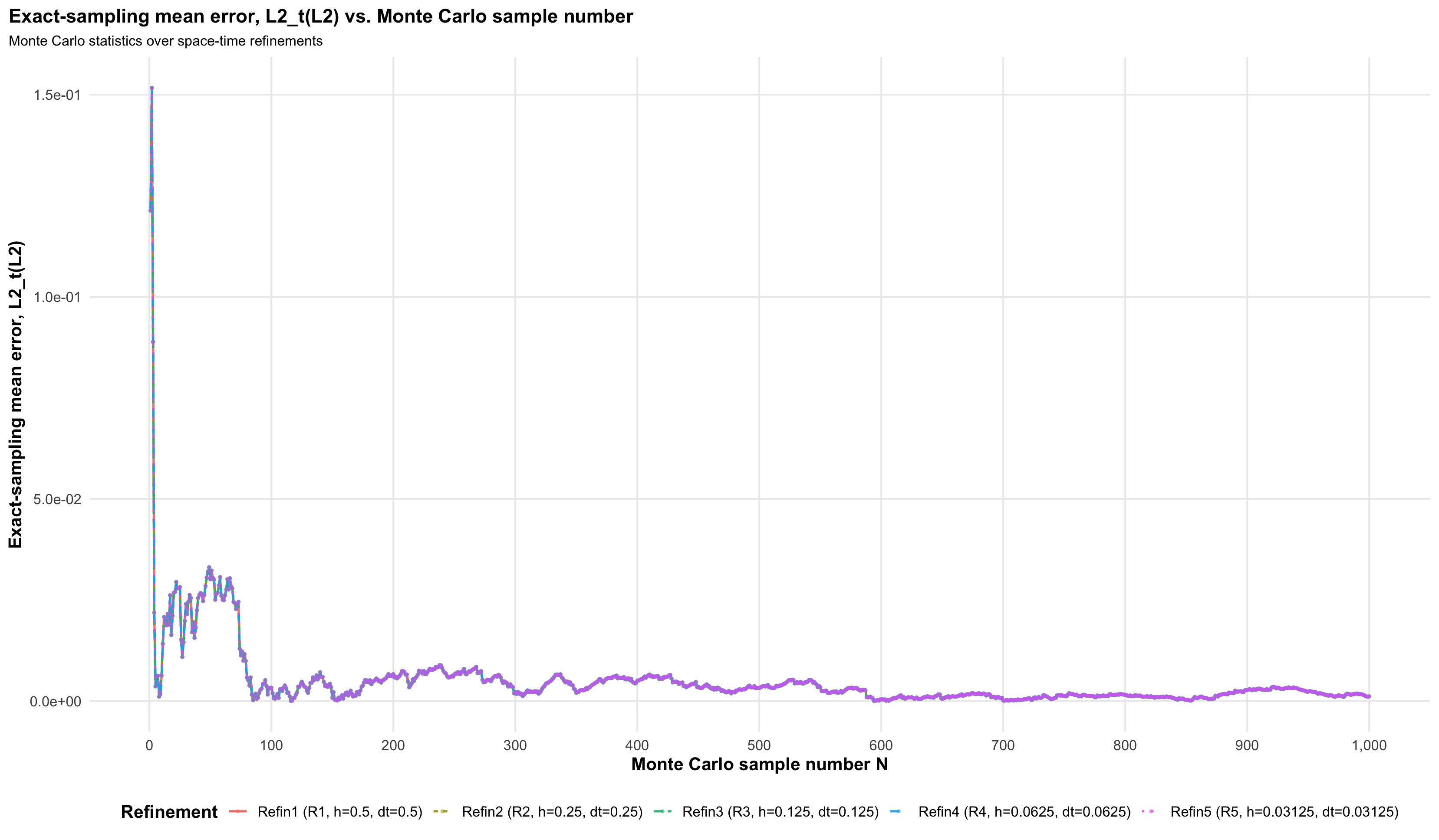}
\caption{Monte-Carlo mean-error diagnostics as a function of the sample number. The curves illustrate the non-monotone but root-sampling-consistent behaviour of a single cumulative Monte-Carlo realization.}
\label{fig:mc-mean-diagnostics}
\end{figure}
\begin{figure}[tbp]
\centering
\includegraphics[width=0.7\linewidth]{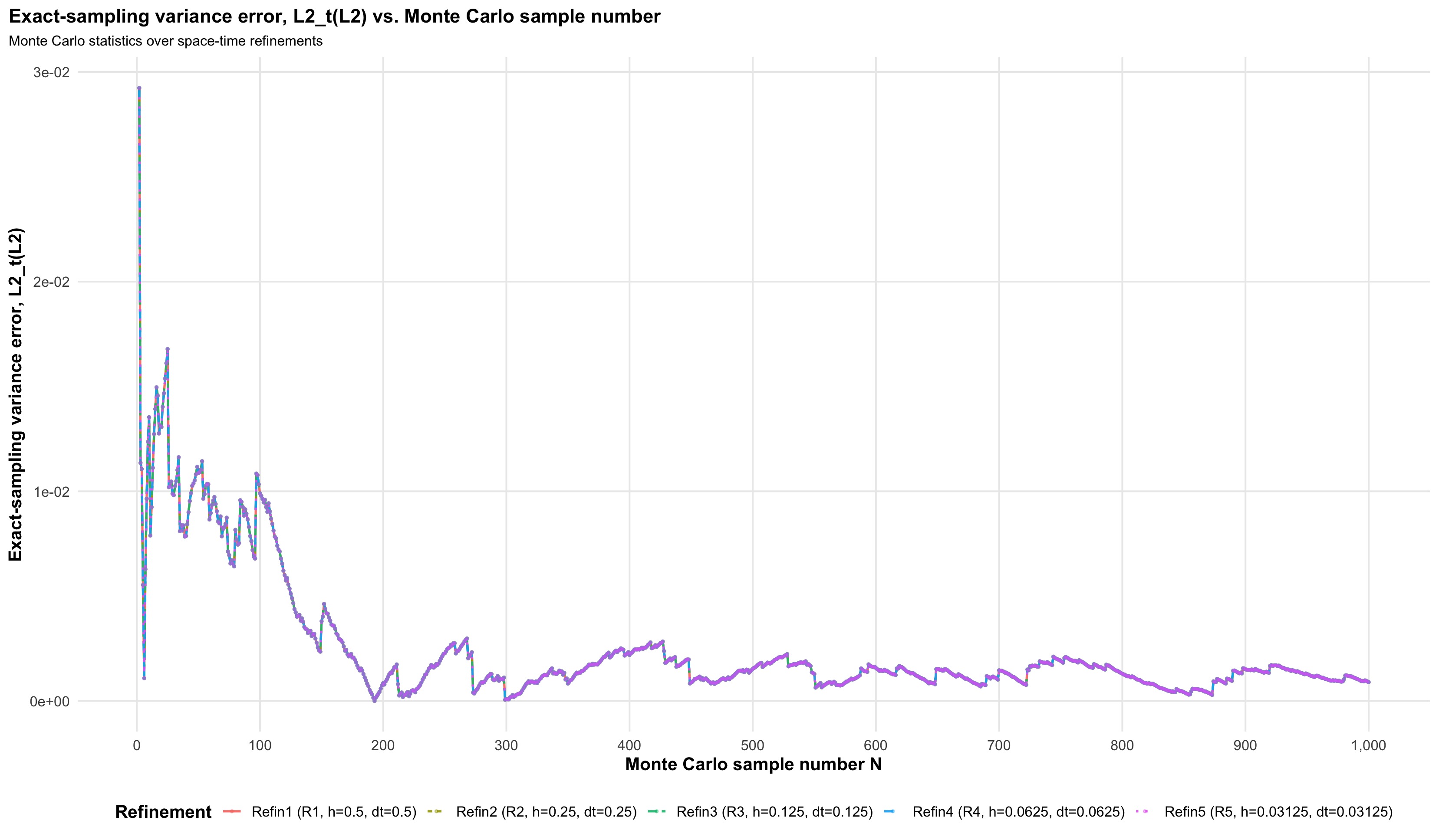}
\caption{Monte-Carlo variance-error diagnostics as a function of the sample number. The variance estimator is more oscillatory than the mean estimator, as expected for a second-moment quantity.}
\label{fig:mc-variance-diagnostics}
\end{figure}

\section{Conclusion: Comparison between stochastic Galerkin and Monte-Carlo methods}
\label{sec:comparison}

We finally compare the stochastic Galerkin and Monte-Carlo results in terms of work and accuracy. The comparison uses the updated data in Table~\ref{tab:mc-sg-complexity-accuracy-walltime} and the combined work-accuracy  plot in Figure~\ref{fig:mc-sg-mean-var-work-accuracy}. The table summarizes selected SG results from stochastic refinement at fixed \(h=\tau=0.03125\), together with space-time refinement at fixed \(p_\xi=5\). The table also includes Monte-Carlo computations with \(N_{\mathrm{MC}}=1000\) samples on several deterministic space-time refinement levels, and an additional experiment with \(N_{\mathrm{MC}}=5000\) computation at \(h=\tau=0.125\).

The figure additionally includes the stochastic-refinement curve at fixed \(h=\tau=0.015625\), so that the two deterministic SG resolutions \(h=\tau=0.03125\) and \(h=\tau=0.015625\) can be compared in the work-accuracy  plane. The plotted Monte-Carlo curves show cumulative sample-count milestones, including the extended \(N_{\mathrm{MC}}=5000\) sequence at \(h=\tau=0.125\). In the plotted SG stochastic-refinement curves, the point \(p_\xi=6\) is omitted in order to focus on the pre-capture behaviour, while Table~\ref{tab:mc-sg-complexity-accuracy-walltime} still reports the \(p_\xi=6=q\) finite-chaos capture row.

All Monte-Carlo and stochastic Galerkin computations reported in this comparison were carried out on the \textbf{HSUper} high-performance computing system at Helmut Schmidt University. The parallel runs used \(5\) compute nodes with \(72\) MPI ranks per node, hence \(360\) MPI ranks in total. For Monte-Carlo, the reported time is the accumulated pathwise wall time \(T_{\mathrm{MC}}^{\mathrm{acc}} = \sum_{m=1}^{N_{\mathrm{MC}}} t_m\), which is used here as the work measure for comparison with the SG wall-clock timings.

\begin{table}
\caption{Comparison of stochastic Galerkin and Monte-Carlo methods for the finite-order prescribed benchmark. The stochastic Galerkin data include both stochastic refinement at fixed \(h=\tau=0.03125\) and space-time refinement at fixed \(p_\xi=5\). The Monte-Carlo rows use \(N_{\mathrm{MC}}=1000\) samples with deterministic space-time refinements \(h=\tau\in\{0.25,0.125,0.0625,0.03125\}\). In addition, the table includes the new \(N_{\mathrm{MC}}=5000\) Monte-Carlo result at \(h=\tau=0.125\). For Monte-Carlo, the reported wall time is the accumulated pathwise wall time
\(T_{\mathrm{MC}}^{\mathrm{acc}}=\sum_{m=1}^{N_{\mathrm{MC}}}t_m\).}
\label{tab:mc-sg-complexity-accuracy-walltime}

\scriptsize
\setlength{\tabcolsep}{2.4pt}
\renewcommand{\arraystretch}{1.12}

\resizebox{\linewidth}{!}{%
\begin{tabular}{@{}llrrrlp{6.2cm}@{}}
\hline
Method
& Resolution
& stochastic size
& \(E^{\mathrm{mean},L^2}\)
& \(E^{\mathrm{var},L^2}\)
& wall time [s]
& main observation \\
\hline
SG
& \(p_\xi=2,\ h=\tau=0.03125\)
& \(N_\xi=15\)
& \(3.77096\times10^{-4}\)
& \(2.97674\times10^{-3}\)
& \(2.67420\times10^{2}\)
& cheapest listed SG stochastic-refinement row; mean and variance remain truncation-influenced. \\

&
\(p_\xi=3,\ h=\tau=0.03125\)
& \(N_\xi=35\)
& \(3.77096\times10^{-4}\)
& \(3.06690\times10^{-4}\)
& \(7.32458\times10^{2}\)
& variance improves substantially; mean remains unchanged from \(p_\xi=2\). \\

&
\(p_\xi=5,\ h=\tau=0.25\)
& \(N_\xi=126\)
& \(3.15222\times10^{-5}\)
& \(1.58571\times10^{-5}\)
& \(9.37018\times10^{1}\)
& coarsest listed \(p_\xi=5\) space-time row; errors already near the \(p_\xi=5\) stochastic plateau. \\

&
\(p_\xi=5,\ h=\tau=0.125\)
& \(N_\xi=126\)
& \(2.32371\times10^{-5}\)
& \(1.45225\times10^{-5}\)
& \(3.30547\times10^{2}\)
& moderate deterministic refinement; mean and variance approach the \(p_\xi=5\) plateau. \\

&
\(p_\xi=5,\ h=\tau=0.0625\)
& \(N_\xi=126\)
& \(2.31988\times10^{-5}\)
& \(1.45194\times10^{-5}\)
& \(1.21895\times10^{3}\)
& further deterministic refinement gives almost no additional error reduction. \\

&
\(p_\xi=5,\ h=\tau=0.03125\)
& \(N_\xi=126\)
& \(2.31986\times10^{-5}\)
& \(1.45194\times10^{-5}\)
& \(3.45131\times10^{3}\)
& finest listed \(p_\xi=5\) grid remains limited by stochastic truncation. \\

&
\(p_\xi=6=q,\ h=\tau=0.03125\)
& \(N_\xi=210\)
& \(5.20063\times10^{-9}\)
& \(1.84941\times10^{-9}\)
& \(6.44612\times10^{3}\)
& finite chaos fully resolved at this deterministic resolution. \\
\hline
MC
& \(h=\tau=0.25\)
& \(N_{\mathrm{MC}}=1000\)
& \(1.122116\times10^{-3}\)
& \(8.971187\times10^{-4}\)
& \(7.14865\times10^{2}\)
& pathwise discretization sufficiently resolved. \\

&
\(h=\tau=0.125\)
& \(N_{\mathrm{MC}}=1000\)
& \(1.121870\times10^{-3}\)
& \(8.969912\times10^{-4}\)
& \(2.86807\times10^{3}\)
& total error unchanged under deterministic refinement. \\

&
\(h=\tau=0.0625\)
& \(N_{\mathrm{MC}}=1000\)
& \(1.121869\times10^{-3}\)
& \(8.969901\times10^{-4}\)
& \(1.51200\times10^{4}\)
& higher deterministic cost with no total-error gain. \\

&
\(h=\tau=0.03125\)
& \(N_{\mathrm{MC}}=1000\)
& \(1.121869\times10^{-3}\)
& \(8.969901\times10^{-4}\)
& \(8.18684\times10^{4}\)
& finest grid remains sampling-limited. \\

&
\(h=\tau=0.125\)
& \(N_{\mathrm{MC}}=5000\)
& \(4.438425\times10^{-4}\)
& \(9.465881\times10^{-5}\)
& \(1.30985\times10^{4}\)
& fivefold sampling reduces the mean error by about \(2.5\) and the variance error by about \(9.5\). \\
\hline
\end{tabular}
}
\end{table}

\paragraph{Observations from Table~\ref{tab:mc-sg-complexity-accuracy-walltime}.}
Table~\ref{tab:mc-sg-complexity-accuracy-walltime} separates three effects: stochastic truncation in the SG chaos space, deterministic space-time discretization, and finite-sample error in Monte-Carlo.

\begin{itemize}
  \item
  In the SG stochastic-refinement rows at fixed \(h=\tau=0.03125\), increasing the chaos degree reduces the variance error very strongly. The variance error decreases from \(2.97674\times10^{-3}\) at \(p_\xi=2\) to \(3.06690\times10^{-4}\) at \(p_\xi=3\), and then to \(1.84941\times10^{-9}\) at \(p_\xi=6=q\). The last value corresponds to the finite-chaos capture point at this deterministic resolution.

  \item
  The SG mean error does not decrease at every stochastic-refinement step. In particular, the mean error is the same at \(p_\xi=2\) and \(p_\xi=3\), namely \(3.77096\times10^{-4}\). This behaviour follows from the modal structure of the coupled SG system. The mean is the zeroth chaos coefficient, but its discrete equation is coupled to higher stochastic modes through the random diffusion coefficient. Since the diffusion coefficient is generated by a squared Gaussian expression, the coupling has an even-mode structure. Hence, some stochastic enrichments improve the variance substantially while leaving the mean essentially unchanged.

  \item
  In the SG space-time-refinement rows at fixed \(p_\xi=5<q\), refining \(h=\tau\) from \(0.25\) to \(0.03125\) reduces the mean error from \(3.15222\times10^{-5}\) to \(2.31986\times10^{-5}\). The variance error decreases from \(1.58571\times10^{-5}\) to \(1.45194\times10^{-5}\). Thus, both quantities quickly reach a \(p_\xi=5\) stochastic-truncation plateau. Further deterministic refinement cannot remove the remaining error caused by the unresolved \(p_\xi=6\) stochastic modes.

  \item
  In the \(N_{\mathrm{MC}}=1000\) rows, refining the pathwise grid from \(h=\tau=0.25\) to \(h=\tau=0.03125\) increases the accumulated wall time from \(7.14865\times10^{2}\) seconds to \(8.18684\times10^{4}\) seconds. However, the total mean error remains approximately \(1.12\times10^{-3}\), and the total variance error remains approximately \(8.97\times10^{-4}\). Thus, for \(N_{\mathrm{MC}}=1000\), the displayed MC errors are sampling-limited rather than pathwise-discretization-limited.

  \item
  The additional \(N_{\mathrm{MC}}=5000\) run at \(h=\tau=0.125\) confirms this interpretation. Keeping the deterministic discretization fixed and increasing the sample number from \(1000\) to \(5000\) reduces the mean error from \(1.121870\times10^{-3}\) to \(4.438425\times10^{-4}\). This is an improvement factor of approximately \(2.53\), which is close to the expected Monte-Carlo factor \(\sqrt{5000/1000}\approx2.24\). The variance error decreases from \(8.969912\times10^{-4}\) to \(9.465881\times10^{-5}\), corresponding to an improvement factor of approximately \(9.48\). This stronger variance reduction is favourable for this particular sample realization, but the systematic conclusion remains that increasing \(N_{\mathrm{MC}}\), not refining \(h\) and \(\tau\), is the effective way to reduce the MC error once the pathwise discretization is sufficiently resolved.

 \item Consequently, an efficient Monte-Carlo workflow should first determine the coarsest space-time resolution at which the total MC mean and variance errors, defined in \eqref{eq:mc-error-mean-L2,mc-error-var-L2}, become sampling-dominated. This can be done by running a few deterministic refinements with a moderate sample number, say \(N_{\mathrm{MC}}=500\) or \(N_{\mathrm{MC}}=1000\). Once further refinement
  of \(h\) and \(\tau\) no longer changes the total errors appreciably, the pathwise discretization is sufficiently resolved. Beyond this point, the computational budget should be invested in increasing \(N_{\mathrm{MC}}\), since this is the only effective way to reduce the remaining sampling error.
\end{itemize}

\begin{figure}[tbp]
\centering
\includegraphics[width=0.9\linewidth]{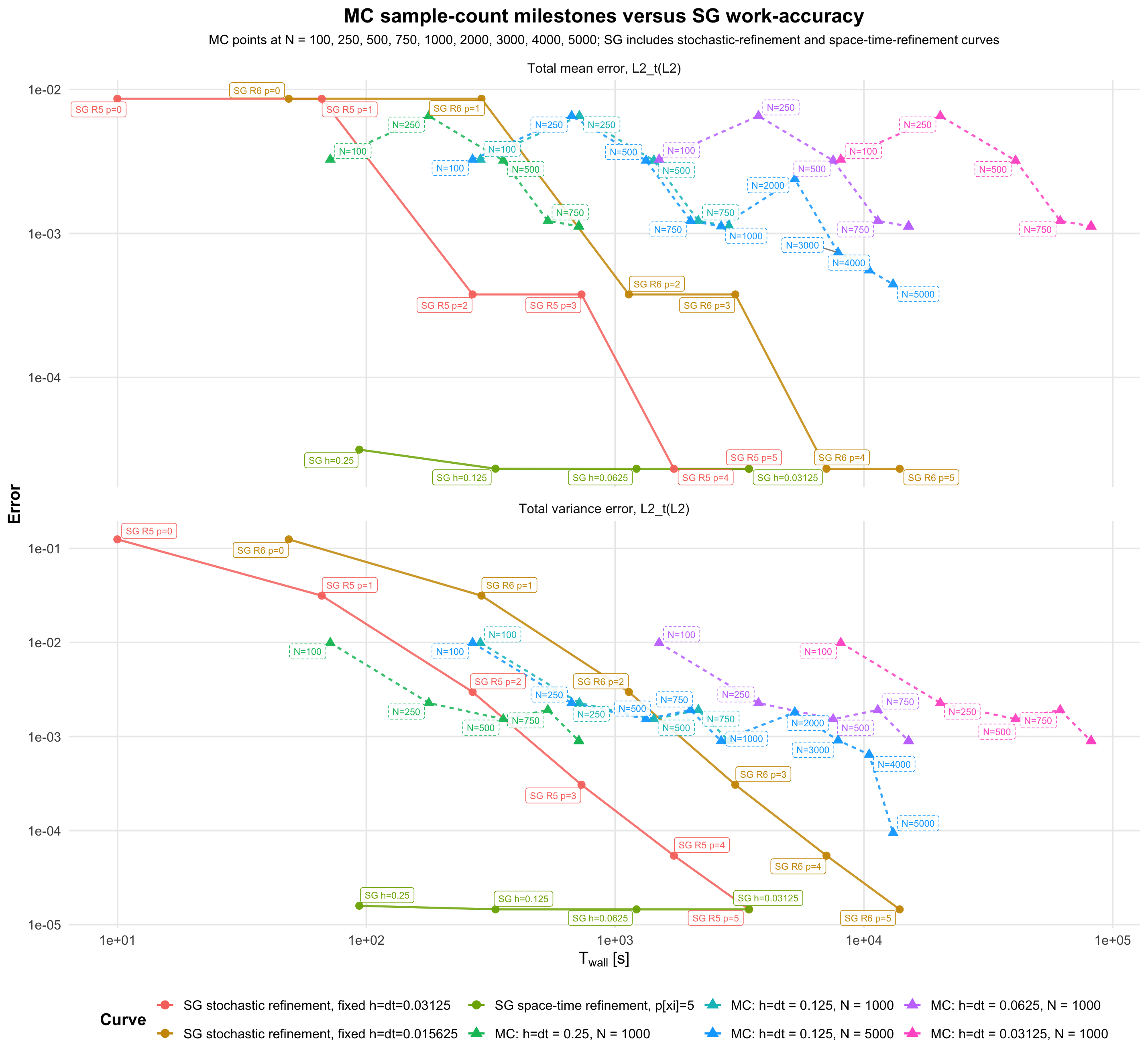}
\caption{Combined work-accuracy  comparison for mean and variance errors. The upper panel shows the total mean error in \(L_t^2(L^2)\), and the lower panel shows the total variance error in \(L_t^2(L^2)\). The Monte-Carlo curves show cumulative estimators at the sample milestones \(N=100,200,400,600,800,1000\). The SG curves show stochastic refinement at fixed \(h=\tau=0.03125\), stochastic refinement at fixed \(h=\tau=0.015625\), and space-time refinement at fixed \(p_\xi=5\). The \(p_\xi=6\) SG capture points are not included in this plot.}
\label{fig:mc-sg-mean-var-work-accuracy}
\end{figure}

\paragraph{Observations from Figure~\ref{fig:mc-sg-mean-var-work-accuracy}.}
Figure~\ref{fig:mc-sg-mean-var-work-accuracy} visualizes the same conclusions as Table~\ref{tab:mc-sg-complexity-accuracy-walltime}.

\begin{itemize}
  \item
  In the upper panel, the \(N_{\mathrm{MC}}=1000\) MC curves remain in a sampling-error band of roughly \(10^{-3}\) to \(10^{-2}\). Increasing the pathwise resolution changes the accumulated wall time, but it does not systematically reduce the displayed MC mean error at the shown sample milestones. This confirms that deterministic refinement mainly shifts the MC points to the right in the work-accuracy  plane.

  \item
  The extended \(N_{\mathrm{MC}}=5000\) curve at \(h=\tau=0.125\) shows the expected effect of increasing the sample number. The final \(N=5000\) point lies below the corresponding \(N=1000\) point for the same deterministic refinement, reducing the mean error from about \(1.12\times10^{-3}\) to about \(4.44\times10^{-4}\). Thus, the figure confirms visually that the MC improvement comes from increasing \(N_{\mathrm{MC}}\), not from further pathwise refinement.

  \item
  The SG stochastic-refinement curves show a more structured error reduction. For both deterministic resolutions, the mean error is nearly unchanged from \(p_\xi=0\) to \(p_\xi=1\), drops at \(p_\xi=2\), remains almost unchanged at \(p_\xi=3\), and drops again at \(p_\xi=4\) and \(p_\xi=5\). This pairwise behaviour is consistent with the even-mode coupling induced by the squared Gaussian diffusion coefficient.

  \item
  The SG space-time-refinement curve at fixed \(p_\xi=5\) is almost horizontal in both panels. This confirms that, at \(p_\xi=5<q\), the error is already dominated by stochastic truncation. Deterministic refinement alone therefore cannot significantly reduce either the mean or the variance error.

  \item
  In the lower panel, the variance error decreases more visibly under stochastic refinement. The SG stochastic-refinement curves reduce the variance error by several orders of magnitude between \(p_\xi=0\) and \(p_\xi=5\). The fixed-\(p_\xi=5\) space-time curve, however, remains near \(1.45\times10^{-5}\), again indicating the stochastic-truncation plateau.

  \item
  The \(N_{\mathrm{MC}}=5000\) Monte-Carlo curve also shows a clear variance-error reduction compared with the \(N_{\mathrm{MC}}=1000\) endpoint at the same deterministic resolution. Nevertheless, even after increasing the sample count to \(5000\), the MC errors remain above the best listed SG errors. The figure therefore illustrates the central difference between the two methods: MC convergence is governed by the slow sampling rate, whereas SG can exploit the finite polynomial stochastic structure of the prescribed analytical solution.
\end{itemize}

Overall, for this finite-order benchmark, stochastic Galerkin provides a more systematic and controllable convergence mechanism for both the mean and the variance as the Hermite chaos space is enriched. At fixed \(p_\xi=5<q\), both SG mean and variance errors reach a stochastic-truncation plateau. When the chaos degree is increased to \(p_\xi=6=q\), as reported in Table~\ref{tab:mc-sg-complexity-accuracy-walltime}, the finite Hermite chaos is fully resolved and the remaining SG error is reduced to the deterministic/algebraic floor. Monte-Carlo remains a robust non-intrusive reference method, but for the present runs with \(N_{\mathrm{MC}}=1000\), its total error is dominated by finite-sample effects and therefore does not improve under further pathwise space-time refinement.

\section*{Acknowledgment}

Computational resources (HPC-cluster HSUper) and funding of M.\ Dawor have been provided by the project hpc.bw of the cluster dtec.bw --- Digitalization and Technology Research Center of the Bundeswehr. dtec.bw is funded by the European Union ---  NextGenerationEU.

\end{document}